\documentclass[12pt,a4,reqno]{amsart}
\usepackage{placeins}
\usepackage{colonequals}
\usepackage{booktabs}

\usepackage{colonequals}

\usepackage{amssymb,amsmath,amsthm,newlfont,enumerate}
\usepackage{bbold}
\usepackage{a4wide}
\usepackage{enumitem}

\usepackage{tikz}
\usetikzlibrary{positioning}
\usepackage[
colorlinks, hyperindex, plainpages=false, bookmarksopen, bookmarksnumbered, pdfusetitle,  linkcolor=blue, citecolor=blue ]{hyperref}

\newtheorem{corollary}{Corollary}[section]

\newcommand{\cD}{{\mathcal{D}}}

\newcommand{\cF}{{\mathcal{F}}}
\newcommand{\J}{\mathbf{J}}
\newcommand{\E}{\mathcal{E}}
\newcommand{\m}{{\mathcal{M}}^+(\TT)}

\renewcommand{\Re}{\mathrm{Re }}

\newcommand{\hH}{\mathrm{H^2 }}

\newcommand{\IL}{\mathrm{L^2(\TT)}}

\newcommand{\Hol}{\mathrm{Hol}}

\newcommand\ZZ{\mathbb{Z}}
\newcommand\TT{\mathbb{T}}
\newcommand\DD{\mathbb{D}}

\DeclareMathOperator{\supp}{supp}

\usepackage{enumitem}
\setlist[itemize]{topsep=0pt,itemsep=2pt,parsep=0pt,partopsep=0pt}
\makeatother

\newtheorem{theorem}{Theorem}[section]
\newtheorem{definition}{Definition}[section]
\newtheorem{remark}{Remark}[section]

\newtheorem{lemma}{Lemma}[section]
\newtheorem{problem}{Problem}

\newtheorem{proposition}{Proposition}[section]

\title[Douglas Weighted Dirichlet Spaces]{Douglas Weighted Dirichlet Spaces: Analytic and Probabilistic Aspects}
\author{J. Bourabiaa, Y. Elmadani, A. Hanine, A. Jamal, I. Labghail}
\date{}
\address{Laboratory of Mathematical Analysis and Applications,
	Mohammed V University in Rabat,
	B.P. 1014, Rabat, Morocco}
\email{jaouad.bourabiaa@gmail.com}
\email{elmadanima@gmail.com}
\email{abhanine@gmail.com}
\email{jamalaymane11@gmail.com}
\email{labghail.imane@gmail.com}
\begin{document} 
	\maketitle
\begin{abstract}
	We introduce a class of weighted Dirichlet spaces on the unit disk, called \emph{Douglas weighted Dirichlet spaces}, characterized by the existence of a Douglas--type boundary representation of the weighted Dirichlet integral. This representation naturally induces a non--local Dirichlet form on the unit circle in the sense of Beurling--Deny and Fukushima. Our main  result shows that the reproducing kernel of the weighted Dirichlet space associated with a regular Douglas weighted Dirichlet form is the image of the Szegő kernel under the corresponding $1$-resolvent. This resolvent representation is new for superharmonic weights. We also develop the potential theory associated with  Douglas weighted Dirichlet spaces. We characterize the capacity in terms of reproducing kernels. Finally, as an application of the fact that every Douglas weight induces a regular non--local Dirichlet form, together with the general theory of Dirichlet forms, we obtain an associated symmetric pure--jump Hunt process. In the classical Dirichlet case, this process is the wrapped Cauchy process.
\end{abstract}
	\tableofcontents

\section{Introduction}	
	Weighted Dirichlet spaces form a natural meeting point for complex analysis, operator theory, potential theory, and stochastic analysis; see, for example, \cite{diaconis2002gh}. They arise in connection with boundary regularity, reproducing kernels, cyclicity, and invariant subspaces, and provide a unified framework for capacities, Dirichlet forms, and symmetric Markov processes. In recent years, Dirichlet forms, also referred to as Dirichlet energies, have become an important tool in graph representation learning, particularly in the theory of graph neural networks; see \cite{wu2021comprehensive}.
	
	Historically, the study of Dirichlet spaces associated with harmonic weights was initiated by Richter \cite{richter1991representation} in connection with the theory of shift-invariant subspaces of the classical Dirichlet space. Later, Aleman \cite{aleman1993multiplication} introduced the broader class of superharmonic weights and established several of their fundamental properties. This framework encompasses an important family of weighted Dirichlet spaces generated by power-type weights, which form a continuous scale interpolating between the classical Dirichlet space and the Hardy space. These spaces also have important applications, notably in connection with approximation problems leading to explicit zero-free regions for the Riemann zeta function; see \cite{gallardo2025zero}. More broadly, several problems related to (super)harmonically weighted Dirichlet spaces, such as boundary behavior, invariant subspaces, zero sets, and Carleson measures, have been studied in the literature; see for instance \cite{ARSW, bahajji2020approximation, bao2018dirichlet, el2016cyclicity, el2022extremal, el2024ahlfors, el2009brown, elmadani2019cyclicity, guillot2012fine, idrissi2020blaschke}.  
	
	One of the fundamental tools in the theory of Dirichlet spaces on the unit disk is the Douglas formula, which represents the Dirichlet integral of a holomorphic function by means of a boundary quadratic form involving differences of its boundary values on the unit circle. This leads naturally to the notion of \emph{Douglas weights}, and we refer to the corresponding weighted Dirichlet spaces as \emph{Douglas weighted Dirichlet spaces}. For such weights, the associated boundary energy is described by a symmetric jump measure on the unit circle, and the resulting quadratic form is a symmetric Dirichlet form in the sense of Beurling, Deny, and Fukushima.
	
	The paper develops both the analytic and probabilistic aspects of Douglas weighted Dirichlet spaces.
	
	First we prove that the harmonic Douglas weighted Dirichlet space is a reproducing kernel Hilbert space and establish an explicit relation between its reproducing kernel and that of the corresponding weighted Dirichlet space. More precisely, if \(K\) denotes the reproducing kernel of the weighted Dirichlet space, then the reproducing kernel of the harmonic space is
	\(
	2\operatorname{Re}K-1.
	\)
	This extends the positivity phenomenon established by Shimorin \cite{shimorin2001reproducing} from superharmonic weights to the class of Douglas weights. We further prove that the reproducing kernel of the Dirichlet space associated with a regular Douglas weighted Dirichlet form is represented as the image of the Szegő kernel under the corresponding $1$-resolvent, yielding a resolvent representation that, to the best of our knowledge, is new even for harmonic and superharmonic weights.
	
	Second, exploiting the underlying Dirichlet form, we introduce the corresponding capacity on subsets of the unit circle and develop its basic properties. In particular, we characterize the capacity through reproducing kernels, obtain an explicit formula for the capacity of singletons in terms of the diagonal values of the reproducing kernel, and represent the capacity of compact sets by means of the reproducing kernel of an associated measurable functional Hilbert space in the sense of Aronszajn et al.~\cite{aronszajn1957characterization}.

	Finally we prove that every Douglas weight generates a regular non--local Dirichlet form. By the general theory of Fukushima \cite{fukushima2010dirichlet}, this yields a symmetric pure--jump Hunt process associated with the Douglas weight. Thus every Douglas weighted Dirichlet space admits a probabilistic realization as the energy space of a jump Markov process. As an illustration, we show that, for the constant weight, the associated Dirichlet form coincides with that of the wrapped Cauchy process; see \cite{bertoin1996levy} for background on Cauchy processes.

	The paper is organized as follows. In Section \ref{section2}, we introduce Douglas weighted Dirichlet spaces through a Douglas--type boundary integral representation. We then establish the associated regular non--local Dirichlet forms and investigate their fundamental properties. In Section \ref{section33}, we study the reproducing kernel Hilbert space associated with these Dirichlet spaces. We prove that the reproducing kernel is obtained as the image of the Szegő kernel under the $1$-resolvent operator. We further characterize the associated capacity in terms of reproducing kernels and develop its potential--theoretic interpretation. In Section \ref{section4}, we prove that every Douglas weight gives rise to a symmetric pure--jump Hunt process. As a particular case, we show that the classical Dirichlet space corresponds to the wrapped Cauchy process. Section \ref{section5} is devoted to several open problems and possible directions for future research. Finally the Appendix collects some auxiliary results.

	\section{Preliminaries} \label{section2}
	
	Let $\mathbb{D}$ denote the open unit disk in the complex plane and $\mathbb{T}$ its boundary, the unit circle. We denote by $\mathrm{Hol}(\mathbb{D})$ the space of all holomorphic functions defined on $\mathbb{D}$. Let $w : \mathbb{D} \to [0,\infty)$ be a measurable function satisfying the integrability condition
	$$
	\int_\mathbb{D} w(z)\, dA(z) < \infty,
	$$ 
	where $dA$ represents the normalized area measure on $\mathbb{D}$. For any $f \in \mathrm{Hol}(\mathbb{D})$, the \emph{weighted Dirichlet integral} is defined as
	$$
	\mathcal{D}_\omega(f) :=\int_D |f'(z)|^2 \omega(z)\, dA(z).
	$$
	The \emph{weighted Dirichlet space}, denoted by $\cD(\omega)$, consists of all functions $f \in \mathrm{Hol}(\mathbb{D})$ for which the weighted Dirichlet integral is finite; specifically, 
	$$
	\cD(\omega):=\left\{f\in \mathrm{Hol}(\DD): \,\, \cD_\omega(f)<\infty\right\}.
	$$
	For  $w \equiv 1$, the space $\cD(\omega)$ coincides with the classical Dirichlet space $\mathcal{D}$.
	
	Recall that the Hardy space $\hH$ is defined by
	$$
	\hH:=\left\{ f\in \mathrm{Hol}(\DD) : f(z)=\sum_{n\ge 0} a_n z^n \ \text{and} \ \|f\|_{\hH}^2:=\sum_{n\ge 0} |a_n|^2<\infty \right\}.
	$$
	For $\omega(z)=\log\!\left(1/|z|^2\right)$, the  space $\cD(\omega)$ coincides with the Hardy space $\hH$, and for every $f\in \hH$ one has
	$$
	\|f\|_{\hH}^2 = |f(0)|^2 + \mathcal D_w(f),
	$$
	see \cite[p.~11--12]{el2014primer}. Moreover for a general weight $\omega$, a sufficient condition for the inclusion $\cD(\omega)\subset \hH$ is that
	$$
	\liminf_{|z|\to 1^-} \frac{\omega(z)}{1-|z|} > 0.
	$$

	We introduce the harmonic analog of the weighted Dirichlet space, denoted by \(\mathcal{D}^h(\omega)\), and defined as follows: 
	$$
	\mathcal{D}^h(\omega) = \left\{ f \in \IL :\, \E_\omega(f,f):=\int_{\mathbb{D}} |\nabla P[f](z)|^2 \omega(z)\, dA(z) < \infty \right\},
	$$	
where $\nabla P[f]$ denotes the gradient of the Poisson extension of $f$. The space $\mathcal{D}^h(\omega)$ is a Hilbert space when equipped with the inner product
	$$
	\langle f,g\rangle_{\mathcal{D}^h(\omega)} = \langle f,g\rangle_{\IL} + \E_\omega(f,g),\qquad f,g\in\cD^h(\omega),
	$$
	where 
	$$\E_\omega(f,g)=\int_\DD \nabla P[f](z)\overline{\nabla P[g](z)} \omega(z)dA(z).$$
	
	It is well known that the weighted Dirichlet integral associated with positive superharmonic weights admits a boundary integral representation of Douglas--type. More precisely, if $\omega$ is a positive superharmonic function on $\DD$, and if $\mu$ and $\nu$ denote, respectively, its Riesz measure and its associated boundary measure, then for every $f\in \IL$, the Poisson extension satisfies
	\begin{equation}\label{Douglas}
    \int_{\DD} |\nabla P[f](z)|^2 \,\omega(z)\, dA(z)
	=
	\int_{\TT}\int_{\TT}
	|f(\zeta)-f(\lambda)|^2
	\left(
	A_\mu(\zeta,\lambda)\, dm(\zeta)
	+
	\frac{d\nu(\zeta)}{|\zeta-\lambda|^2}
	\right)
	dm(\lambda),
	\end{equation}
	where
	$$
	A_\mu(\zeta,\lambda)
	=
	\int_{\DD}
	\frac{1-|z|^2}{|\zeta-z|^2}
	\frac{1-|z|^2}{|\lambda-z|^2}
	\, d\mu(z),
	$$
	and $m$ denotes the normalized Lebesgue measure on $\TT$. A proof of \eqref{Douglas} can be found in \cite{bao2018dirichlet}; see also \cite[p.~221]{koosis1998introduction}. This representation naturally leads to the introduction of a broader class of weights for which the Dirichlet energy admits a purely boundary formulation.
	
	\begin{definition}
		We say that a weight $\omega$ belongs to the Douglas class $\mathcal W_{\mathrm{Dou}}$ if there exists a positive Borel measure $\J_\omega$ on $\TT^2 \setminus \mathrm{diag}$ such that
		$$
		\int_{\DD} |\nabla P[f](z)|^2 \,\omega(z)\, dA(z)
		=
		\int_{\TT}\int_{\TT}
		|f(\zeta)-f(\lambda)|^2
		\, \J_\omega(d\zeta,d\lambda),
		$$  for every $f \in \IL$.
	\end{definition}
	It is worth recalling that Douglas--type inequalities were established in Besov spaces, we refer the reader to \cite{idrissi2025douglas}.

	Let $T:\mathbb{C}\to\mathbb{C}$ be a normalized contraction, that is,
	$$
	|T(\zeta_1)-T(\zeta_2)|\le |\zeta_1-\zeta_2|,
	\qquad \zeta_1,\zeta_2\in\mathbb{C},
	$$
	and $T(0)=0$. Following Beurling and Deny \cite{beurling1958espaces}, a \textit{Dirichlet space} is defined as a Hilbert space $D$ of complex--valued functions on $\mathbb{T}$, locally integrable with respect to $m$, and satisfying the following axioms:
	\begin{enumerate}[label=\roman*)]
		\item For each compact set $K\subset\mathbb{T}$, there exists a constant $A(K)>0$ such that
		$$
		\int_K |f(\zeta)|\,dm(\zeta)\le A(K)\|f\|,
		\qquad f\in D.
		$$
		
		\item For every normalized contraction $T$ of $\mathbb{C}$ and every $f\in D$, one has
		$$
		Tf\in D
		\quad\text{and}\quad
		\|Tf\|\le \|f\|.
		$$
	\end{enumerate}
	
	We denote by $C(\mathbb T)$ the space of continuous functions on $\mathbb T$. The Dirichlet space $D$ is said to be \textit{regular} if the subspace
	$D\cap C(\mathbb T)$ is dense in $D$ with respect to the Dirichlet norm and dense in $C(\mathbb T)$ with respect to the uniform norm.
	\begin{proposition}\label{douglas-dirichlet}
		Let $\omega \in \mathcal W_{\mathrm{Dou}}$. Then the harmonic Dirichlet space
		$\mathcal{D}^h(\omega)$ is a Dirichlet space in the sense of Beurling--Deny.
	\end{proposition}
	We also say that $\mathcal E_\omega$ is a Dirichlet form in the sense of Beurling--Deny. In the particular case where $\omega\in\mathcal W_{\mathrm{Dou}}$, we refer to $\mathcal E_\omega$ as a \textit{Douglas weighted Dirichlet form}.
	\begin{proof}
		Let $\omega \in \mathcal W_{\mathrm{Dou}}$. By definition of this class, there exists a positive Borel measure
		$ \J_\omega$ on $\mathbb{T}^2 \setminus \mathrm{diag}$ such that for every $f \in \mathcal{D}^h(\omega)$,
		$$
		\int_{\mathbb{D}} |\nabla P[f](z)|^2 \omega(z)\, dA(z)
		=
		\int_{\mathbb{T}}\int_{\mathbb{T}}
		|f(\zeta)-f(\lambda)|^2 \, \J_\omega(d\zeta,d\lambda).
		$$
		Consequently the norm on $\mathcal{D}^{h}(\omega)$ can be written as
		$$
		\|f\|_{\mathcal{D}^h(\omega)}^2
		=
		\|f\|_{\IL}^2
		+
		\int_{\mathbb{T}}\int_{\mathbb{T}}
		|f(\zeta)-f(\lambda)|^2\, \J_\omega(d\zeta,d\lambda).
		$$
		Hence functions in $\mathcal{D}^h(\omega)$ belong to $\IL$, and for every compact set
		$K \subset \mathbb{T}$, we have by the Cauchy--Schwarz inequality
		$$
		\int_K |f(\zeta)|\, dm(\zeta)
		\le C(K)\|f\|_{\IL}
		\le C'(K)\|f\|_{\mathcal{D}^h(\omega)},
		$$
		which establishes axiom (i).
		
		We now verify axiom (ii). Let $f \in \mathcal{D}^h(\omega)$. Then
		$$
		|Tf(\zeta)-Tf(\lambda)|
		\le
		|f(\zeta)-f(\lambda)|, \quad  \zeta,\lambda\in\mathbb{T}.
		$$
		Therefore
		$$
		\int_{\mathbb{T}}\int_{\mathbb{T}}
		|Tf(\zeta)-Tf(\lambda)|^2\, \J_\omega(d\zeta,d\lambda)
		\le
		\int_{\mathbb{T}}\int_{\mathbb{T}}
		|f(\zeta)-f(\lambda)|^2\, \J_\omega(d\zeta,d\lambda).
		$$
		Moreover
		$$
		\|Tf\|_{\IL} \le \|f\|_{\IL}.
		$$
		Combining the preceding inequalities, we obtain
		$$
		\|Tf\|_{\mathcal{D}^h(\omega)}
		\leq
		\|f\|_{\mathcal{D}^h(\omega)},
		$$
		which implies that \(Tf\in\mathcal{D}^h(\omega)\).
		Thus $\mathcal{D}^h(\omega)$ satisfies both axioms of Beurling–Deny and is therefore a Dirichlet space.
	\end{proof}

	Next, we derive a representation formula for Douglas weighted Dirichlet forms. To this end, we recall the notions needed for the proof and adopt the terminology of \cite[Theorem~1.4.1]{fukushima2010dirichlet}.
	
	Let $\mu$ be a Radon measure on $X$ and let $(X,\mathcal{B},\mu)$ be a $\sigma$-finite measure space. A linear operator
	$$
	S:\mathrm{L}^2(X,\mu)\longrightarrow \mathrm{L}^2(X,\mu)
	$$
	with $\mathcal D(S)=\mathrm{L}^2(X,\mu)$ is called \emph{Markovian} if
	$$
	0 \leq Su \leq 1 \quad m\text{-a.e.},
	$$
	whenever $u\in\mathrm{L}^2(X,\mu)$ satisfies
	$$
	0 \leq u \leq 1 \quad m\text{-a.e.}
	$$
	In particular, $S$ is positivity preserving, that is,
	$$
	u\geq 0 \quad\Longrightarrow\quad Su\geq 0
	\quad m\text{-a.e.}
	$$
	for every $u\in\mathrm{L}^2(X,\mu)$.
	
	The proof of the representation formula relies on the following lemma, whose proof can be found in \cite[Lemma~1.4.1]{fukushima2010dirichlet}.

	\begin{lemma}\label{Markovian}
		Let $S$ be a positive symmetric linear operator on $\mathrm{L}^2(X,\mu)$.
		
		\begin{enumerate}
			\item[(i)] There exists a unique positive Radon measure $\nu$ on $X\times X$ such that for all Borel functions $f,g \in \mathrm{L}^2(X,\mu)$,
			\begin{equation*}\label{eq:kernel_representation}
				\langle f,Sg\rangle_{\mathrm{L}^2(X,\mu)} = \int_{X\times X} f(x)g(y)\,\nu(dx,dy).
			\end{equation*}
			
			\item[(ii)] If in addition $S$ is Markovian, then
			\begin{equation*}\label{eq:markov_measure}
				\nu(X\times E) \leq \mu(E), \quad \forall E \in \mathcal{B}(X).
			\end{equation*}
		\end{enumerate}
	\end{lemma}

	We are  in a position to state the representation formula for Douglas weighted Dirichlet forms.

	\begin{proposition}\label{TheoremD}
		Let \(\omega\in\mathcal{W}_{\mathrm{Dou}}\). Then there exists a unique Radon measure \(\sigma_\alpha^{(\omega)}\) on \(\mathbb{D} \times \mathbb{D}\) such that
		$$
		\sigma_\alpha^{(\omega)}(\mathbb{D} \times E) \leq A(E), \quad E \subset \mathbb{D}.
		$$
		Moreover define
		$$
		s_\alpha^{(\omega)} = \frac{d\sigma_\alpha^{(\omega)}(\cdot, z)}{dA(z)}, \quad s_\alpha^{(\omega)} \in [0, 1].
		$$
		For any \(f \in \mathcal{D}^h(\omega)\), we have
		$$
		\E_\omega(f,f) = \lim_{\alpha \to +\infty}  \frac{\alpha}{2} \int_{\mathbb{D} \times \mathbb{D}} |P[f](z) - P[f](w)|^2 \, d\sigma_\alpha^{(\omega)}(z, w) + \alpha\int_{\mathbb{D}} |P[f](z)|^2 (1 - s_\alpha^{(\omega)}(z)) \, dA(z).
		$$
	\end{proposition}

	\begin{proof}
		Let $(\mathcal{R}_\alpha)_{\alpha>0}$ be the resolvent associated with $\E_\omega$ on $\mathrm{L}^2(\mathbb{D},dA)$. For each $\alpha>0$, the operator $\alpha \mathcal{R}_\alpha$ is a positive symmetric contraction on $\mathrm{L}^2(\mathbb{D},dA)$ and is Markovian. By Lemma~\ref{Markovian}, there exists a unique positive Radon measure $\sigma_\alpha^{(\omega)}$ on $\mathbb{D}\times\mathbb{D}$ such that
		\begin{equation}\label{eq:resolvent_measure}
			\alpha \langle f,\mathcal{R}_\alpha g\rangle_{\mathrm{L}^2(\DD,dA)}
			=
			\int_{\mathbb{D}\times\mathbb{D}} f(z)\overline{g(w)}\,d\sigma_\alpha^{(\omega)}(z,w),
			\quad f,g \in \mathrm{L}^2(\mathbb{D},dA).
		\end{equation}
		Moreover since $\alpha \mathcal{R}_\alpha$ is Markovian, we have
		\begin{equation}\label{eq:measure_bound}
			\sigma_\alpha^{(\omega)}(\mathbb{D}\times E)\le A(E),
			\qquad E\subset\mathbb{D}.
		\end{equation}

		From \eqref{eq:measure_bound}, the measure $\sigma_\alpha^{(\omega)}(\mathbb{D}\times\cdot)$ is absolutely continuous with respect to $dA$. Hence there exists a measurable function $s_\alpha^{(\omega)}$ such that
		$$
		s_\alpha^{(\omega)}(z)=\frac{d\sigma_\alpha^{(\omega)}(\mathbb{D}\times dz)}{dA(z)}.
		$$
		In particular,
		$$
		0\le s_\alpha^{(\omega)}(z)\le 1, \quad A-\text{a.e. }
		$$
		For $f\in \mathrm{L}^2(\mathbb{D},dA)$, using \eqref{eq:resolvent_measure} with $g=f$, we obtain
		$$
		\alpha \langle f,\mathcal{R}_\alpha f\rangle_{\mathrm{L}^2(\DD,dA)}
		=
		\int_{\mathbb{D}\times\mathbb{D}} P[f](z)\overline{P[f](w)}\,d\sigma_\alpha^{(\omega)}(z,w).
		$$
		Moreover
		$$
		\|f\|_{\mathrm{L}^2(\DD,dA)}^2
		=
		\int_{\mathbb{D}} |P[f](z)|^2\,dA(z).
		$$
		A direct computation yields
		\begin{align*}
			\|f\|_{\mathrm{L}^2(\DD,dA)}^2 - \alpha \langle f,\mathcal{R}_\alpha f\rangle_{\mathrm{L}^2(\DD,dA)}
			&=
			\frac{1}{2}
			\int_{\mathbb{D}\times\mathbb{D}} |P[f](z)-P[f](w)|^2\,d\sigma_\alpha^{(\omega)}(z,w)
			\\
			&\quad
			+
			\int_{\mathbb{D}} |P[f](z)|^2\bigl(1-s_\alpha^{(\omega)}(z)\bigr)\,dA(z).
		\end{align*}
		Multiplying by $\alpha$, we obtain
		\begin{align}\label{eq:approximation_form}
			\alpha\bigl(\|f\|_{\mathrm{L}^2(\DD,dA)}^2 - \alpha\langle f,\mathcal{R}_\alpha f\rangle_{\mathrm{L}^2(\DD,dA)}\bigr)
			&=
			\frac{\alpha}{2}
			\int_{\mathbb{D}\times\mathbb{D}} |P[f](z)-P[f](w)|^2\,d\sigma_\alpha^{(\omega)}(z,w)
			\nonumber\\
			&\quad
			+
			\alpha \int_{\mathbb{D}} |P[f](z)|^2\bigl(1-s_\alpha^{(\omega)}(z)\bigr)\,dA(z).
		\end{align}
		By \cite[Lemma~1.3.4]{fukushima2010dirichlet}, for every $f \in \mathcal{D}^h(\omega)$,
		$$
		\E_\omega(f,f)
		=
		\lim_{\alpha\to\infty}
		\alpha \bigl(\|f\|_{\mathrm{L}^2(\DD,dA)}^2 - \alpha \langle f,\mathcal{R}_\alpha f\rangle_{\mathrm{L}^2(\DD,dA)}\bigr).
		$$
		Combining this with \eqref{eq:approximation_form}, we conclude that
		\begin{equation*}
			\E_\omega(f,f)=
			\lim_{\alpha\to\infty}
			\frac{\alpha}{2}
			\int_{\mathbb{D}\times\mathbb{D}} |P[f](z)-P[f](w)|^2\,d\sigma_\alpha^{(\omega)}(z,w)+
			\alpha \int_{\mathbb{D}} |P[f](z)|^2\bigl(1-s_\alpha^{(\omega)}(z)\bigr)\,dA(z)
		\end{equation*}
		This completes the proof.
	\end{proof}
	The following proposition provides a partial converse to Proposition~\ref{douglas-dirichlet}.
	\begin{proposition}\label{TheoremK}
		Let $\omega$ be a weight on the unit disk $\mathbb{D}$. If the bilinear form
		$\E_{\omega}$  defined by
		$$
		\E_{\omega}(f,g) = \int_{\mathbb{D}} \nabla P[f](z)\,  \overline{\nabla P[g](z)}\, \omega(z)\, dA(z),\qquad f,g\in\IL,
		$$ is a regular Dirichlet form on $\IL$, then
		$  \omega \in \mathcal{W}_{\mathrm{Dou}}.$
	\end{proposition}

	\begin{proof}
		Suppose that $\E_{\omega}$ is a regular Dirichlet form on $\IL$. By the Beurling--Deny decomposition theorem
		\cite[Theorem~3.2.1, p.~120]{fukushima2010dirichlet}, it admits the decomposition
		$$
		\E_{\omega}(f,g)=\E_{\omega}^{(c)}(f,g)+\E_{\omega}^{(J)}(f,g)+\E_{\omega}^{(k)}(f,g),
		$$
		where $\E_{\omega}^{(c)}$, $\E_{\omega}^{(J)}$, and $\E_{\omega}^{(k)}$ denote the strongly local, jump, and killing parts, respectively.\\
		We claim that
		$$
		\E_{\omega}^{(c)} \equiv 0,\qquad\E_{\omega}^{(k)} \equiv 0.
		$$
		First, since the constant function $1$ belongs to $\mathcal{D}(\E_{\omega})$, we have
		$$
		0=\E_{\omega}(1,1)=\E_{\omega}^{(k)}(1,1)=\int_{\mathbb{T}} k(d\zeta)=k(\mathbb{T}),
		$$
		where we used the fact that the strongly local and jump parts vanish on constant functions. Since $k$ is a positive measure, it follows that $k\equiv 0$
		and hence
		$$
		\E_{\omega}^{(k)} \equiv 0.
		$$
		Next, we show that the strongly local part vanishes. Let $f_n(\zeta)=\zeta^n$ for  $n\in\mathbb{N}$, we have $P[f_n](z)=z^n,$ and therefore $|\nabla P[f_n](z)|^2=n^2|z|^{2(n-1)}.$ Consequently
		$$
		\E_{\omega}(f_n,f_n)=n^2\int_{\mathbb{D}}|z|^{2(n-1)}\omega(z)\,dA(z).
		$$
		On the other hand, by the representation theorem \cite[Theorem~3.2.3, p.~130]{fukushima2010dirichlet} for strongly local Dirichlet forms on $\mathbb{T}$, there exists a positive Borel measure $\nu$ on $\mathbb{T}$ such that
		$$
		\E_{\omega}^{(c)}(u,u)=
		\int_{\mathbb{T}} |u'(\zeta)|^2\,\nu(d\zeta),\quad  u\in \mathcal D(\E_{\omega})\cap C^1(\mathbb{T}),
		$$
		we obtain
		$$
		\E_{\omega}^{(c)}(f_n,f_n) =n^2\nu(\mathbb{T}).
		$$
		Moreover
		$\E_{\omega}^{(c)}(f_n,f_n)
		\le\E_{\omega}(f_n,f_n)$,
		which implies
		$$
		\nu(\mathbb{T})
		\le
		\int_{\mathbb{D}}
		|z|^{2(n-1)}
		\omega(z)\,dA(z),
		\qquad n\in\mathbb{N}.
		$$
		Since $\omega\in \mathrm L^{1}(\mathbb{D},dA)$ and  $0\le |z|^{2(n-1)}\omega(z)\le \omega(z)$ while $|z|^{2(n-1)} \longrightarrow 0 \quad \text{pointwise on } \mathbb{D}$, the dominated convergence theorem implies that
		$$
		\lim_{n\to\infty}\int_{\mathbb{D}}|z|^{2(n-1)}\omega(z)\,dA(z)=0.
		$$
		Hence $\nu(\mathbb{T})=0$ which implies that
		$$
		\E_{\omega}^{(c)}\equiv 0.
		$$
		Therefore
		$\E_{\omega}=\E_{\omega}^{(J)}$, 
		that is, $\omega \in \mathcal{W}_{\mathrm{Dou}}$.
	\end{proof}

	 \section{Analytic aspect } \label{section33}

	\subsection{Reproducing kernel}
	Throughout this paper, we assume that
	\begin{align}\label{condition1}
		\cD(\omega) \subset \hH.
	\end{align} 
	In this case, equipped with the norm
	$$
	\|f\|_{\omega}^2 = \|f\|_{\mathrm H^2}^2 + \mathcal{D}_\omega(f),
	$$
	$\cD(\omega)$ becomes a reproducing kernel Hilbert space. Denoting by $k^\omega$ its reproducing kernel, we have
	$$
	f(z) = \langle f, k^\omega_z \rangle_\omega, \quad z \in \DD
	$$
	for every $f \in \cD(\omega)$.
	
	Several important properties of reproducing kernels have been established for weighted Dirichlet spaces associated with superharmonic weights. In \cite{shimorin2002complete}, Shimorin proved that the reproducing kernel $k^\omega$ of $\mathcal D(\omega)$, where $\omega$ is a superharmonic weight on $\DD$, is a Nevanlinna--Pick kernel. Consequently there exists a positive semidefinite function $B(z,\lambda)$ on $\DD\times\DD$, with $B(\cdot,0)=0$, such that
	$$
		k^\omega(z,\lambda)=\frac{1}{1-B(z,\lambda)},\qquad z,\lambda\in\DD.
	$$
	In the harmonic setting, El-Fallah, Elmadani, and Kellay \cite{el2019kernel} established the following sharp estimate:
	\begin{equation}\label{kernelestimate}
		k^{\omega}(z,z)\asymp 1+\int_{0}^{|z|}\frac{dr}{(1-r)\omega(rz/|z|)+(1-r)^2}, \quad z \in \DD.
	\end{equation}
	where the implicit constants are absolute. This estimate was subsequently extended to the class of superharmonic weights in \cite{bahajji2026superharmonically}.
	
	The following special cases highlight the versatility and the limiting behavior of the $\mathcal{D}^h(\omega)$ framework:
	
	\noindent
	\textbf{(1) The case $\omega\equiv0$.}
	The space $\mathcal D^h(0)$ coincides with the classical harmonic Hardy space
	$h^2(\mathbb D)$ (see, e.g., \cite[Chapter~V]{axler2001harmonic}). More precisely, every function $u\in\mathcal D^h(0)$ is the Poisson extension of a unique function $f\in \IL$, namely
	$$
	u(z)=P[f](z)=\int_{\mathbb T}f(\zeta)
	\frac{1-|z|^2}{|\zeta-z|^2}
	dm(\zeta),
	\qquad z\in\mathbb D.
	$$
	Since
	$$
	\frac{1-|z|^2}{|\zeta-z|^2}
	=2\Re k^0(\zeta,z)-1,\quad \zeta\in\TT,\,z\in\DD,
	$$
	where $k^0$ denotes the Szeg\H{o} kernel of $\hH$, it follows that
	$$
	u(z)
	=
	\int_{\mathbb T}
	f(\zeta)\bigl(2\Re k^0(\zeta,z)-1\bigr)
	dm(\zeta)
	=
	\bigl\langle
	f,\,
	2\Re k^0(\cdot,z)-1
	\bigr\rangle_{\IL}.
	$$
	Consequently $h^2(\mathbb{D})$ is a Hilbert space with reproducing kernel given by
	
	$$
	2\,\Re k^0(z,w)-1,\quad z,w\in\DD.
	$$
	
	\noindent
	\textbf{(2) The case $\omega\equiv1$.} The space $\mathcal{D}^h(\omega)$ coincides with the classical harmonic Dirichlet space $\mathcal{D}(\mathbb{T})$. Its norm admits the Fourier characterization
	$$
	\|f\|^2_{\mathcal{D}(\mathbb{T})}
	=\sum_{n\in\mathbb{Z}}(1+|n|)\,|\hat{f}(n)|^2,
	\qquad f\in\IL.
	$$
	For further details, we refer the reader to
	\cite{abakumov2015cyclicity, ross1994hyperinvariant}.
	
	The reproducing kernel of $\mathcal{D}$ is given by
	$$
	k^1(z,w)
	=\sum_{n=0}^{\infty}\frac{(z\overline{w})^n}{n+1}
	=\frac{1}{z\overline{w}}\log\!\left(\frac{1}{1-z\overline{w}}\right).
	$$
	Moreover its real part satisfies the integral representation
	$$
	2\Re k^1(z,w)-1
	=\int_{0}^{1}\frac{1-r^2|z\overline{w}|^2}{|1-rz\overline{w}|^2}\,dr
	=\int_{0}^{1} P(rz,w)\,dr \ge 0,
	$$
	which coincides with the reproducing kernel of the harmonic Dirichlet space
	$\mathcal{D}(\mathbb{T})$.

	The structural properties exhibited by these two examples extend to general weights $\omega$. In the next proposition, we establish a reproducing kernel representation for  $\mathcal{D}^h(\omega)$.

	\begin{proposition}\label{prop:reproducing_kernel}
		Let $\omega$ be a weight satisfying Condition~\eqref{condition1}. Then the space $\mathcal{D}^h(\omega)$ is a reproducing kernel Hilbert space, and every $f \in \mathcal{D}^h(\omega)$ admits the representation
		$$
		P[f](z) = \langle f,\, 2\Re k^\omega(\cdot,z) - 1 \rangle_{\cD^h(\omega)}, \quad z \in \mathbb{D}.
		$$
	\end{proposition}

	\begin{proof}

		Let $\mathbf P_\omega$ be the Riesz projection associated with
		$\mathcal D^h(\omega)$. We first establish that $\mathbf P_\omega$
		commutes with the harmonic extension operator. Specifically, 
		$$
		P[\mathbf P_\omega f]=\mathbf P_\omega P[f],
		\qquad f\in \IL.
		$$
		Indeed, writing
		$$
		f(e^{it})=\sum_{n\in\mathbb Z}\widehat f(n)e^{int},
		$$
		we have
		$$
		P[f](z)
		=
		\sum_{n\geq0}\widehat f(n)z^n
		+
		\sum_{n\geq1}\widehat f(-n)\overline z^{\,n}.
		$$
		Since $\mathbf P_\omega$ extracts the holomorphic component of a
		harmonic function, it follows that
		$$
		\mathbf P_\omega P[f](z)
		=
		\sum_{n\geq0}\widehat f(n)z^n
		=
		P[\mathbf P_\omega f](z).
		$$
		Now, the Wirtinger derivative $\partial$ vanishes on the
		antiholomorphic component of a harmonic function. Hence
		$$
		\partial P[\mathbf P_\omega f]=\partial P[f],
		\qquad
		\partial P[\mathbf P_\omega g]=\partial P[g], \quad f,g\in\IL.
		$$
		Therefore
		$$
		\begin{aligned}
			\mathcal E_\omega(\mathbf P_\omega f,g)
			&=
			2\int_{\mathbb D}
			\partial P[\mathbf P_\omega f]\,
			\overline{\partial P[g]}\,
			\omega\,dA  \\
			&=
			2\int_{\mathbb D}
			\partial P[f]\,
			\overline{\partial P[g]}\,
			\omega\,dA \\
			&=
			2\int_{\mathbb D}
			\partial P[f]\,
			\overline{\partial P[\mathbf P_\omega g]}\,
			\omega\,dA \\
			&=
			\mathcal E_\omega(f,\mathbf P_\omega g).
		\end{aligned}
		$$
		Thus $\mathbf P_\omega$ is self--adjoint with respect to  $\mathcal E_\omega$.
		
		We next consider the decomposition induced by $\mathbf P_\omega$. For
		$f\in\mathcal D^h(\omega)$, let
		$$
		f^+=\mathbf P_\omega f,
		\qquad
		f^-=f-\mathbf P_\omega f .
		$$
		Then
		$$
		f^+\in\mathcal D(\omega),
		\qquad
		f^-\in\overline{\mathcal D_0(\omega)},
		$$
		where
		$$
		\mathcal D_0(\omega)
		=
		\{f\in\mathcal D(\omega):f(0)=0\}.
		$$
		Moreover the orthogonality of the two components follows from the
		self--adjointness of $\mathbf P_\omega$:
		$$
		\langle f^+,f^-\rangle_{\mathcal D^h(\omega)}
		=
		\langle \mathbf P_\omega f,f-\mathbf P_\omega f\rangle_{\mathcal D^h(\omega)}
		=0.
		$$
		Consequently
		$$
		\mathcal D^h(\omega)
		=
		\mathcal D(\omega)\oplus\overline{\mathcal D_0(\omega)} .
		$$
		It remains to identify the reproducing kernel of
		$\mathcal D^h(\omega)$. Since $\mathbf 1_{\mathbb D}\in\mathcal D(\omega)$,
		we have
		$$
		\langle f,\mathbf 1_{\mathbb D}\rangle_\omega
		=
		\langle f,\mathbf 1_{\mathbb D}\rangle_{\hH}
		+
		\mathcal D_\omega(f,\mathbf 1_{\mathbb D})
		=
		f(0).
		$$
		Hence
		$$
		k^\omega(\cdot,0)=\mathbf 1_{\mathbb D}.
		$$
		Applying the projection $\mathbf P_\omega$ to $\overline{k^\omega(\cdot,z)}$ gives
		$$
		\mathbf P_\omega\overline{k^\omega(\cdot,z)}
		=
		\mathbf 1_{\mathbb D}.
		$$
		The preceding identity, together with the orthogonal decomposition of
		$\mathcal D^h(\omega)$, allows us to identify its reproducing kernel as
		$$
		2\,\Re k^\omega(\cdot,z)-1 .
		$$
	\end{proof}

	The following positivity result was established by Shimorin in \cite{shimorin2001reproducing}. His argument also extends to the vector-valued setting.
	
\begin{theorem}
Let $\omega$ be a superharmonic weight on $\DD$. Then 
 \begin{equation}\label{positive}
2\Re\, k^\omega(z,w)-1\geq 0,
\qquad z,w\in\DD.
\end{equation} 
\end{theorem}

	A natural question arising from this result is whether the positivity property \eqref{positive} remains valid for more general weights. In this paper, we show that \eqref{positive} continues to hold for every weight in the class $\mathcal W_{\mathrm{Dou}}$.
	\begin{theorem}\label{Theo1}
		Let $\omega\in\mathcal W_{\mathrm{Dou}}$ satisfy Condition~\eqref{condition1}. Then the reproducing kernel $k^{\omega}$ satisfies
		$$
		2\Re\,k^\omega(z,w) - 1 \geq 0, \quad z,w \in \mathbb{D}.
		$$
	\end{theorem}
	The proof of Theorem~\ref{Theo1} relies on the following fundamental result.
	
	\begin{lemma}\cite[Theorem 1]{aronszajn1957characterization}\label{Lemma6}
		Let \((\mathcal{F}, \lVert \cdot \rVert)\) be a proper functional Hilbert space with reproducing kernel \(K\). 
		The kernel \(K\) is nonnegative if and only if \(\mathcal{F}\) satisfies:
		\begin{enumerate}
			\item \label{ass1} closure under complex conjugation:  For all \(f \in \mathcal{F}\), 
			\(\overline{f} \in \mathcal{F}\) and \(\lVert \overline{f} \rVert = \lVert f \rVert\).
			
			\item \label{ass2} Dominating function property: For every real-valued \(f \in \mathcal{F}\), 
			there exists \(\widetilde{f} \in \mathcal{F}\) satisfying
			$$
			\widetilde{f}(x) \geq |f(x)| \quad \text{for all } x \quad \text{and} \quad 
			\lVert \widetilde{f} \rVert \leq \lVert f \rVert.
			$$
		\end{enumerate}
	\end{lemma}
	
	\begin{proof}[Proof of Theorem \ref{Theo1}]
		Let $\omega \in \mathcal W_{\mathrm{Dou}}$. Then $\cD^h(\omega)$ satisfies the assumptions of Lemma \ref{Lemma6}. Property \eqref{ass1} is immediate from the definition of $\cD^h(\omega)$, while \eqref{ass2} follows from the Beurling--Deny property applied to $\tilde{f}=|f|$ for $f \in \cD^h(\omega)$.
	\end{proof}

	\begin{remark}
		By a result of Deny \cite[Lemma~1, p.~142]{brelot2011potential}, every
		Dirichlet space in the sense of
		\cite{fukushima2010dirichlet} has a nonnegative kernel. The notion of kernel
		used there belongs to potential theory and does not refer, in general, to the
		reproducing kernel of a Hilbert space.
	\end{remark}
	Let $\E_1$ denote the classical Douglas--Dirichlet form on $\IL$ and let
	$$
	\mathcal{G}_1:D(\mathcal{G}_1)\subset \IL\longrightarrow \IL
	$$
	be its infinitesimal generator, defined by
	$$
	(\mathcal{G}_1f)(\zeta)
	=
	2\,\mathrm{PV}\!\int_{\TT}
	\frac{f(\zeta)-f(\lambda)}
	{|\zeta-\lambda|^2}
	\,dm(\lambda),
	\qquad \zeta\in\TT,
	$$
	where $\mathrm{PV}$ denotes the Cauchy principal value. The associated $1$-resolvent is given by
	$$
	\mathcal{R}_1=(I+\mathcal{G}_1)^{-1}.
	$$
	The Fourier basis $\{e^{in\theta}\}_{n\in\ZZ}$ forms a complete system of eigenfunctions for both $\mathcal{G}_1$ and $\mathcal{R}_1$. More precisely,
	$$
	\mathcal{G}_1(e^{in\theta})
	=
	|n|\,e^{in\theta},
	\qquad
	\mathcal{R}_1(e^{in\theta})
	=
	\frac{1}{1+|n|}
	e^{in\theta},
	\qquad n\in\ZZ.
	$$
	The first identity follows from a direct computation based on the Fejér kernel, whereas the second follows immediately from the functional calculus for the nonnegative self--adjoint operator $\mathcal{G}_1$.
	
	Let
	$$
	k_z^0(\zeta)
	=
	\frac{1}{1-\overline{z}\zeta}
	=
	\sum_{n=0}^{\infty}r^n(\overline{\lambda}\zeta)^n,
	\qquad
	z=r\lambda\in\DD,\;\zeta\in\TT,
	$$
	denote the Szegő kernel. Applying the resolvent $\mathcal{R}_1$ term by term yields
	\begin{align*}
		\mathcal{R}_1k_z^0(\zeta)
		&=
		\sum_{n=0}^{\infty}
		r^n\,\mathcal{R}_1\big((\overline{\lambda}\zeta)^n\big)\\
		&=
		\sum_{n=0}^{\infty}
		\frac{r^n(\overline{\lambda}\zeta)^n}{n+1}\\
		&=
		\frac{1}{\overline{z}\zeta}
		\log\frac{1}{1-\overline{z}\zeta}.
	\end{align*}
	The right-hand side coincides with the reproducing kernel of $\cD$. Consequently
	$$
	k_z^1
	=
	\mathcal{R}_1k_z^0,
	\qquad z\in\DD.
	$$
	This representation shows that the reproducing kernel of the classical Dirichlet space is obtained by applying the $1$-resolvent of the associated Dirichlet form to the Szegő kernel. The following theorem establishes that this phenomenon extends to the Douglas weighted Dirichlet spaces associated with admissible weights in the class $\mathcal W_{\mathrm{Dou}}$.
	
	\begin{theorem}
		Let $\omega\in\mathcal W_{\mathrm{Dou}}$ satisfy Condition~\eqref{condition1}. If $(\E_\omega,\cD^h(\omega))$ is a regular Dirichlet form, then
		$$
		k^\omega(\cdot,z)=\mathcal{R}_\omega k^0(\cdot,z), \qquad z\in\DD,
		$$
		where $\mathcal{R}_\omega$ denotes the $1$-resolvent of $\E_\omega$.
	\end{theorem}

	\begin{proof}
		Assume that $\omega \in \mathcal{W}_{\mathrm{Dou}}$. By the representation theorem for closed symmetric forms, there exists a unique nonnegative self--adjoint operator
		$\mathcal{G}_\omega$ on $\IL$ such that
		$$
		\E_{\omega}(f,g) =\langle \mathcal{G}_\omega f,g\rangle_{\IL}, \qquad f\in \mathcal D(\mathcal{G}_\omega),\; g\in \mathcal D(\E_{\omega}).
		$$
		By Proposition~\ref{prop:reproducing_kernel}, we have
		$$
		P[f](z) = \langle f,\,2\Re k^\omega(\cdot,z)-1\rangle_{\cD^h(\omega)},\quad z\in \DD.
		$$
		Therefore
		\begin{align*}
			P[f](z)
			&=
			\left\langle f,2\Re k^\omega(\cdot,z)-1 \right\rangle_{\IL}+\E_{\omega}\!\left(f,2\Re k^\omega(\cdot,z)-1\right)
			\\
			&=\left\langle f, 2\Re k^\omega(\cdot,z)-1 \right\rangle_{\IL}+\left\langle \mathcal{G}_\omega f, 2\Re k^\omega(\cdot,z)-1 \right\rangle_{\IL}.
		\end{align*}
		Since $\mathcal{G}_\omega$ is self--adjoint, it follows that
		\begin{align*}
			P[f](z) &=
			\left\langle f, 2\Re k^\omega(\cdot,z)-1 \right\rangle_{\IL}+
			\left\langle f, \mathcal{G}_\omega\!\left(2\Re k^\omega(\cdot,z)-1\right) \right\rangle_{\IL}\\
			&=\left\langle f, (I+\mathcal{G}_\omega) \left(2\Re k^\omega(\cdot,z)-1\right) \right\rangle_{\IL}.
		\end{align*}
		On the other hand, $P[f](z)=\langle f,\,2\Re k^0(\cdot,z)-1\rangle_{\IL}$, for every $f\in\IL$. Hence
		$$
		(I+\mathcal{G}_\omega)\left(2\Re k^\omega(\cdot,z)-1\right)=2\Re k^0(\cdot,z)-1.
		$$
		Since  $\mathcal{G}_\omega(1)=0$, we deduce that
		$$
		(I+\mathcal{G}_\omega)\Re k^\omega(\cdot,z)=\Re k^0(\cdot,z).
		$$
		Moreover since $I+\mathcal{G}_\omega$ is a linear operator,
		$$
		\Re\!\left((I+\mathcal{G}_\omega)k^\omega(\cdot,z)\right)=\Re k^0(\cdot,z).
		$$
		Finally both functions $(I+\mathcal{G}_\omega)k^\omega(\cdot,z)$ and $k^0(\cdot,z)
		$
		are holomorphic on $\mathbb{D}$ and have the same real part. Hence their difference is a purely imaginary constant. Since
		$$
		\bigl((I+\mathcal{G}_\omega)k^\omega(\cdot,z)\bigr)(0)=k^0(0,z)=1,
		$$
		this constant must vanish. Consequently
		$$
		(I+\mathcal G_\omega)k^\omega(\cdot,z)=k^0(\cdot,z).
		$$
		Moreover $\mathcal G_\omega$ is a nonnegative self--adjoint operator on $\IL$, then the operator $I+\mathcal G_\omega$ is invertible. Hence
		$$
		k^\omega(\cdot,z)
		=(I+\mathcal G_\omega)^{-1}k^0(\cdot,z).
		$$
	\end{proof}
	When $\omega$ is superharmonic, it is well known that the trigonometric polynomials are dense in $\cD^h(\omega)$ (see, e.g., \cite{aleman1993multiplication,el2016dirichlet}). Therefore $C(\TT)\cap\cD^h(\omega)$ is dense both in $\cD^h(\omega)$ and in $C(\TT)$, and it follows that $(\E_\omega,\cD^h(\omega))$ is a regular Dirichlet form. We thus obtain the following result. 
	\begin{corollary}\label{corokernel}
		Let $\omega$ be a superharmonic weight on $\DD$. The reproducing kernel of $\cD(\omega)$ satisfies
		$$
		k^\omega(\cdot,z)=\mathcal{R}_\omega k^0(\cdot,z), \qquad z\in\DD,
		$$
		where $\mathcal{R}_\omega$ is the $1$-resolvent associated with $\E_\omega$.
	\end{corollary}
	\begin{remark}
		For the particular weight $\omega=P[\delta_1]$, the connection between weighted Dirichlet spaces and De Branges--Rovnyak spaces was first established by Sarason \cite{Sarason1997} and was subsequently developed by Chevrot, Guillot and Ransford \cite[Theorem~4.1]{ChevrotGuillotRansford2010}; see also \cite{el2016dirichlet}. These works show that the weighted Dirichlet space $\cD(P[\delta_1])$ is isometrically equal to a De Branges--Rovnyak space. Consequently its reproducing kernel is the corresponding De Branges--Rovnyak kernel.
		
		On the other hand, Corollary~\ref{corokernel} shows that the same reproducing kernel is obtained by applying the $1$-resolvent associated with the Dirichlet form to the Szegő kernel, namely
		$$
		k^{P[\delta_1]}(\cdot,z)
		=
		\mathcal R_{P[\delta_1]}k^0(\cdot,z),
		\qquad z\in\DD.
		$$
		Thus our representation provides a new operator-theoretic interpretation of the De Branges--Rovnyak kernel in this particular case. More importantly, unlike the De Branges--Rovnyak representation, which is presently available only for weights of the form $P[\delta_\zeta]$, the resolvent formula extends to every superharmonic weight.
	\end{remark}
\subsection{Foundations of Capacity }\label{section3} In this section, we recall several fundamental properties of the capacity associated with the Dirichlet space $\mathcal{D}(\omega)$. While most of these properties are well established in the case of harmonic weights, their extensions to the general setting considered here can be found in \cite{fukushima2010dirichlet}.

	Throughout this section, we assume that the weight $\omega$ belongs to $\mathcal{W}_{\mathrm{Dou}}$.

	\begin{definition}
		Let $\omega \in \mathcal{W}_{\mathrm{Dou}}$. We associate with the harmonic Dirichlet space $\mathcal{D}^{h}(\omega)$ the capacity $c_{\omega}$ defined as follows. For every open subset $O \subset \mathbb{T}$, we set
		$$
		c_{\omega}(O)
		=
		\inf \left\{
		\|f\|_{\mathcal{D}^{h}(\omega)}^{2}
		:
		f \in \mathcal{D}^{h}(\omega),
		\ |f| \geq 1
		\ \text{m-a.e. on } O
		\right\}.
		$$
		The capacity of an arbitrary subset $A \subset \mathbb{T}$ is then defined by outer regularization:
		$$
		c_{\omega}(A)
		=
		\inf \left\{
		c_{\omega}(O)
		:
		A \subset O,\,
		O \subset \mathbb{T}\ \text{open}
		\right\}.
		$$   
	\end{definition}
	
	This Choquet--type capacity satisfies Fukushima’s axiomatic framework \cite{fukushima2010dirichlet}. The associated quasi--notions, such as quasi--everywhere convergence and quasi--continuity, follow standard definitions.
	
	We say that a property holds $c_\omega$-quasi--everywhere ($c_\omega$-q.e.) if it holds outside a set of $c_\omega$-capacity zero. A function $f\in \IL$ is called $c_\omega$-quasi--continuous if, for every $\epsilon>0$, there exists an open set $O\subset\TT$ such that $c_\omega(O)<\epsilon$ and the restriction of $f$ to $\TT\setminus O$ is continuous.
	
	A function $g$ is called a $c_\omega$-quasi--continuous modification of $f$ if $g$ is $c_\omega$-quasi--continuous and $g=f$ $m$-a.e. on $\TT$.
	
	We denote by $\widetilde{f}$ the $c_\omega$-quasi--continuous modification of $f$.

	Recall that, in the classical setting, the notion of capacity is intrinsically linked to the Riesz and logarithmic kernels. The following proposition provides an analogous characterization in the setting of weighted Dirichlet spaces.
	
	\begin{proposition}\label{proposition41}
		
		Let $\omega \in \mathcal{W}_{\mathrm{Dou}}$ and let $K$ be a closed subset of $\TT$. If $c_\omega(K)>0$, then there exists a unique measure
		$
		\mu_K\in\m,
		$
		whose support is contained in $K$, such that
		\begin{equation}\label{Condition 2}
			c_{\omega}(K)
			=\|p_{\mu_K}\|_{\mathcal{D}^{h}(\omega)}^{2}
			=\mu_K(K),
		\end{equation}
		where $p_{\mu_K}$ denotes the potential associated with the measure $\mu_K$.
		
		Moreover $p_{\mu_K}$ satisfies the following properties.
		\begin{enumerate}
			\item	$
			0\leq p_{\mu_K}\leq 1
			\text{ on }\TT,
			$
			and
			$
			\widetilde{p_{\mu_K}}=1
			\quad
			c_{\omega}\text{-q.e. on }K.
			$
			
			\item
			For every $f\in\mathcal{D}^{h}(\omega)$,
			$$
			\langle p_{\mu_K},f\rangle_{\mathcal{D}^{h}(\omega)}
			=
			\int_{\TT}\widetilde{f}(\xi)\,d\mu_K(\xi).
			$$
			
			\item
			For every $z\in\DD$,
			$$
			P[p_{\mu_K}](z)
			=
			\int_{\TT}
			\widetilde{\bigl(2\Re\,k^{\omega}(z,\lambda)-1\bigr)}
			\,d\mu_K(\lambda).
			$$
		\end{enumerate}
	\end{proposition}

	\begin{proof}
		The arguments leading to $(1)$ and $(2)$ can be found in a more general framework in 
		\cite[Theorem 2.1.5 and Theorem 2.2.2]{fukushima2010dirichlet}. 
		For completeness we reproduce the main ideas.
		
		Consider
		$$
		S(K):=\{f\in \cD^h(\omega): f\ge0 \text{ and } \widetilde f\ge1 
		\; c_{\omega}\text{-q.e. on } K\}.
		$$
		This set is convex and closed in $\cD^h(\omega)$. Hence there exists a unique 
		function $g_K\in\cD^h(\omega)$ with $g_K\ge0$ such that 
		$\widetilde{g_K}\ge1$ $c_{\omega}$-q.e. on $K$ and
		$$
		c_{\omega}(K)=\|g_K\|_{\mathcal{D}^{h}(\omega)}^2.
		$$
		Define $p_K=\min(g_K,1)$. Then
		$$
		c_{\omega}(K)=\|p_K\|_{\mathcal{D}^{h}(\omega)}^2,
		$$
		and the function $p_K$ satisfies $0\le p_K\le1$ together with 
		$\widetilde{p_K}=1$ $c_{\omega}$-q.e. on $K$, which establishes assertion $(1)$.
		
		To prove assertion $(2)$, let $f\in\cD^h(\omega)$ be nonnegative. 
		For any $\varepsilon>0$, the function $p_K+\varepsilon f$ belongs to $\cD^h(\omega)$. 
		Therefore
		$$
		2\langle p_K,f\rangle_{\mathcal{D}^{h}(\omega)}+\varepsilon\|f\|_{\mathcal{D}^{h}(\omega)}^2\ge0.
		$$
		Letting $\varepsilon\to0$ yields $\langle p_K,f\rangle_{\mathcal{D}^{h}(\omega)}\ge0$ for every 
		nonnegative $f\in\cD^h(\omega)$. 
		This implies the existence of a unique positive Borel measure 
		$\mu_K\in\m$, supported on $K$, such that
		\begin{equation}\label{rep}
			\langle p_{\mu_K},f\rangle_{\mathcal{D}^{h}(\omega)}
			:=\langle p_K,f\rangle_{\mathcal{D}^{h}(\omega)}
			=\int_{\TT}f(\xi)\,d\mu_K(\xi),
		\end{equation}
		for all $f\in C(\TT)$.
		
		Next we extend \eqref{rep} to functions in $\cD^h(\omega)$. 
		Given $f\in\cD^h(\omega)$, there exists a sequence 
		$f_n\in \cD^h(\omega)\cap C(\TT)$ converging to $f$ in $\cD^h(\omega)$ 
		and a subsequence $f_{n_k}$ converging $c_{\omega}$-q.e. to $\widetilde f$. 
		Applying Fatou's lemma gives
		\begin{align*}
			\int_{\TT}\left|\widetilde f(\zeta)-f_n(\zeta)\right|d\mu_K(\zeta)
			&\le \liminf_{n_k\to\infty}
			\int_{\TT}|f_{n_k}(\zeta)-f_n(\zeta)|\,d\mu_K(\zeta) \\
			&\le \liminf_{n_k\to\infty}
			\|p_{\mu_K}\|_{\mathcal{D}^{h}(\omega)}\|f_{n_k}-f_n\|_{\mathcal{D}^{h}(\omega)}\\
			&=c_{\omega}(K)^{1/2}\liminf_{n_k\to\infty}\|f_{n_k}-f_n\|_{\mathcal{D}^{h}(\omega)}.
		\end{align*}
		Consequently $\widetilde{\cD^h(\omega)}\subset \mathrm L^1(\mu_K)$ and
		$$
		\langle p_{\mu_K},f\rangle_{\mathcal{D}^{h}(\omega)}
		=\int_{\TT}\widetilde f(\zeta)\,d\mu_K(\zeta),
		\qquad f\in\cD^h(\omega).
		$$
		To establish assertion $(3)$ we apply Proposition \ref{prop:reproducing_kernel} to $p_{\mu_K}\in\cD^h(\omega),$ this yields
		\begin{align*}
			P[ p_{\mu_K}](z)
			&=\langle p_{\mu_K},2\Re\,k^{\omega}(z,\cdot)-1\rangle_{\mathcal{D}^{h}(\omega)}\\
			&=\int_{\TT}\widetilde{\big(2\Re\,k^{\omega}(z,\lambda)-1\big)}
			\,d\mu_K(\lambda).
		\end{align*}
		Finally $\mu_K$ is supported on $K$ and
		$$
		\|p_{\mu_K}\|_{\mathcal{D}^{h}(\omega)}^2
		=\langle p_{\mu_K},p_{\mu_K}\rangle_{\mathcal{D}^{h}(\omega)}
		=\int_{\TT}\widetilde{p_{\mu_K}}(\lambda)\,d\mu_K(\lambda)
		=\mu_K(K).
		$$
		
		This completes the proof.
	\end{proof}

	As an immediate consequence of the preceding result, we obtain the following corollary.
	
	\begin{corollary}
		Let $\omega \in \mathcal W_{\mathrm{Dou}}.$ Let $K$ be a closed subset of $\TT$. Then
		$$
		c_{\omega}(K)
		=\sup\{\nu(K):\nu\in\mathcal{M}^+(\TT),\ \supp\nu\subset K,\ 
		\widetilde{p_\nu}\le1\ \ c_{\omega}\text{-q.e. on }\TT\}.
		$$
	\end{corollary}
	In the case where $\omega$ is a harmonic weight on $\mathbb{D}$, it is well known that
	$$
	c_\omega(\{\xi\})>0,\qquad \xi\in\mathbb{T},
	$$
	if and only if 
	$
	\lim_{r\to1^-}P[f](r\xi)
	$
	exists and is finite for every $f\in\mathcal{D}^h(\omega)$. This characterization remains valid for Douglas weights. To investigate the radial boundary behaviour of functions in $\mathcal{D}^h(\omega)$, one may adapt the arguments developed in \cite[p.~21--25]{guillot2010comportement}.
	
	More precisely, the analysis of radial limits relies on the following analogue of \cite[Theorem~28]{guillot2010comportement}.

	\begin{theorem}
		Let $\omega \in \mathcal W_{\mathrm{Dou}}$ and let $\zeta \in \mathbb{T}$. The following are equivalent:
		\begin{enumerate}
			\item $\lim_{r \to 1} P[f](r\zeta)$ exists and is finite for every $f \in \mathcal{D}^h(\omega)$;
			\item $c_\omega(\{\zeta\}) > 0$.
		\end{enumerate}
	\end{theorem}
	
	A related problem is the computation of the capacity $c_\omega$ for standard subsets of $\TT$, such as singleton sets and arcs.
	
	In the case of power weights $\omega=\omega_\alpha$, Stegenga proved in \cite{stegenga1980multipliers} that
	$$
	c_\alpha(I)\asymp
	\begin{cases}
		|I|^\alpha, & \alpha\in(0,1),\\
		\left(\log\frac{1}{|I|}\right)^{-1}, & \alpha=0.
	\end{cases}
	$$
	
	In \cite{el2019kernel}, for harmonic weights $\omega$, the capacity $c_\omega$ was related to the reproducing kernel $k^\omega$ of $\mathcal{D}(\omega)$. More precisely, let $I\subset\mathbb{T}$ be an arc of length $|I|=1-\rho$ and let $\zeta\in\mathbb{T}$ denote its midpoint. Then
	\begin{equation}
		c_\omega(I)\asymp \frac{1}{k^\omega(\rho\zeta,\rho\zeta)},
		\label{capker}
	\end{equation}
	where the implicit constants are independent of $I$. This estimate was recently extended to the class of superharmonic weights in \cite{bahajji2026superharmonically}.
	
	Combining the kernel estimate \eqref{kernelestimate} with \eqref{capker} yields the estimate
	\begin{equation}\label{capacitysingleton}
		c_\omega(\{\zeta\})
		\asymp
		\left(
		1+\int_0^1
		\frac{dr}
		{(1-r)\omega(r\zeta)+(1-r)^2}
		\right)^{-1},
		\qquad \zeta\in\mathbb{T}.
	\end{equation}
	
	In the present work, we obtain an exact formula for the capacity of a singleton.
	
	\begin{theorem}\label{capkernl}
		Let $\omega \in \mathcal{W}_{\mathrm{Dou}}$ satisfy Condition \eqref{condition1} and let $\zeta \in \mathbb{T}$. Then
		$$
		c_\omega(\{\zeta\})
		=
		\frac{1}{2k^\omega(\zeta,\zeta)-1}.
		$$
		In particular,
		$ c_\omega(\{\zeta\})=0
		\text{ if and only if }
		k^\omega(\zeta,\zeta)=\infty.
		$
	\end{theorem}
	To prove Theorem~\ref{capkernl}, we first establish the following estimate.
	
	\begin{proposition}\label{Theoremcaparc}
		Let $\omega \in \mathcal W_{\mathrm{Dou}}$ satisfy Condition \eqref{condition1} and let \( I \subset \TT \) be an arc of length \( |I| = 1 - \rho \) centered at the point \( \zeta \in \TT \). Then
		$$
		c_\omega(I) \geq \frac{C(\rho)}{2k^\omega(\rho \zeta, \rho \zeta) - 1},
		$$
		where
		$$
		C(\rho) = \left( \frac{1 + \rho}{1 + \frac{\rho}{4}} \right)^2.
		$$
	\end{proposition}
	
	\begin{proof}
		Let $f \in \cD^h(\omega)$ satisfy $f \geq 0$ and $f \geq 1$ m-a.e. on $I$. By the Poisson integral representation,
		$$
		P[f](\rho\zeta)
		\geq
		\int_{I}\frac{1-\rho^{2}}{|\lambda-\rho\zeta|^{2}}\,dm(\lambda).
		$$
		A straightforward computation shows that
		$$
		|\lambda-\rho\zeta|^{2}
		\leq
		\left(1+\frac{\rho}{4}\right)|I|^{2}.
		$$
		Consequently
		$$
		P[f](\rho\zeta)
		\geq
		\int_{I}
		\frac{1-\rho^{2}}
		{\left(1+\frac{\rho}{4}\right)|I|^{2}}
		\,dm(\lambda)
		=
		\frac{1+\rho}{1+\frac{\rho}{4}}.
		$$
		Set
		$$
		C(\rho):=
		\left(
		\frac{1+\rho}
		{1+\frac{\rho}{4}}
		\right)^{2}.
		$$
		Then
		$$
		C(\rho)
		\leq
    P[f]^{2}(\rho\zeta)
		\leq
		\|f\|_{\cD^{h}(\omega)}^{2}
		\bigl(2k^{\omega}(\rho\zeta,\rho\zeta)-1\bigr).
		$$
		Since the above estimate holds for every admissible function $f$, taking the infimum over all $f\in\cD^{h}(\omega)$ satisfying $f\ge0$ and $f\ge1$ m-a.e. on $I$ yields
		$$
		C(\rho)
		\leq
		c_{\omega}(I)
		\bigl(2k^{\omega}(\rho\zeta,\rho\zeta)-1\bigr).
		$$
	\end{proof}
	\begin{proof}[Proof of Theorem \ref{capkernl} ]
		Let $\xi \in \mathbb{T}$. Assume first that $c_\omega(\{\xi\})>0$. By Proposition~\ref{proposition41}, there exists an equilibrium measure $\mu_\xi$ supported on $\{\xi\}$  of the form
		$$
		\mu_\xi=c_\xi\delta_\xi,
		$$
		where $\delta_\xi$ denotes the Dirac measure at $\xi$. Moreover
		$$
		c_\omega(\{\xi\})=\|p_{\mu_\xi}\|_{\mathcal{D}^{h}(\omega)}^2=\mu_\xi(\{\xi\})=c_\xi.
		$$
		Since $c_\omega(\{\xi\})>0$, Proposition~\ref{proposition41} also yields
		$$
		\widetilde{p_{\mu_\xi}}(\xi)=1.
		$$
		On the other hand, for every $z\in\mathbb{D}$,
		$$
		P[p_{\mu_\xi}](z)
		=\int_{\mathbb{T}}
		\widetilde{\bigl(2\Re\,k^\omega(z,\lambda)-1\bigr)}
		\,d\mu_\xi(\lambda)
		=c_\xi\bigl(2\Re\,\widetilde{k^\omega(z,\xi)}-1\bigr).
		$$
		Evaluating at $z=\xi$, we obtain
		$$
		c_\xi\bigl(2\Re\,\widetilde{k^\omega(\xi,\xi)}-1\bigr)=1.
		$$
		Consequently
		$$
		c_\xi\bigl(2k^\omega(\xi,\xi)-1\bigr)=1,
		$$
		which proves the desired identity.
		
		Assume now that $c_\omega(\{\xi\})=0$. Let $(I_n)_{n\ge1}$ be a decreasing sequence of open arcs containing $\xi$ of the form
		$$
		I_n=\left(e^{i(\theta-\frac1n)},\,e^{i(\theta+\frac1n)}\right),
		\qquad \xi=e^{i\theta},
		$$
		such that
		$$
		\lim_{n\to\infty}c_\omega(I_n)=0.
		$$
		By Proposition~\ref{Theoremcaparc},
		$$
		c_\omega(I_n)\ge
		\frac{C(\rho_n)}
		{2k^\omega(\rho_n\xi,\rho_n\xi)-1},
		\qquad n\ge1.
		$$
		Equivalently
		$$
		k^\omega(\rho_n\xi,\rho_n\xi)
		\ge
		\frac{1}{2}
		\left(
		\frac{C(\rho_n)}{c_\omega(I_n)}
		+1
		\right).
		$$
		Since $C(\rho_n)\to1$ and $c_\omega(I_n)\to0$ as $n\to\infty$, it follows that
		
		$$
		k^\omega(\xi,\xi)=\infty,
		$$
		which completes the proof.
	\end{proof}

	Furthermore the capacity associated with $\mathcal D^h(\omega)$ admits a natural characterization in terms of the diagonal values of the reproducing kernel of an auxiliary measurable functional Hilbert space.
	
	Let  $x^*\subset\mathbb T$ be a Borel set with positive $c_\omega$-capacity and let $\mu_{x^*}$ denote its equilibrium measure. In the terminology of Aronszajn and Smith \cite{aronszajn1957characterization}, the pair $(\mathcal D^h(\omega),\mu_{x^*})$ is a measurable functional Hilbert space. Denote by $\mathcal B^*$ the family of Borel subsets of $\mathbb T$. The associated functional Hilbert space, denoted by $\mathcal D^h(\omega)^*$, consists of all set functions of the form
	$$
	u^*(x^*)
	=
	\int_{x^*}u\,d\mu_{x^*},
	\qquad x^*\in\mathcal B^*,
	$$
	where $u\in\mathcal D^h(\omega)$. The correspondence
	$
	u\longmapsto u^*
	$
	is injective. Endowing $\mathcal D^h(\omega)^*$ with the inner product
	$$
	\langle u^*,v^*\rangle_*
	=
	\langle u,v\rangle_{\mathcal D^h(\omega)},
	$$
	turns this correspondence into a canonical Hilbert space isomorphism between
	$\mathcal D^h(\omega)$ and $\mathcal D^h(\omega)^*$.
	
	The next proposition shows that $\mathcal D^h(\omega)^*$ is itself a reproducing kernel Hilbert space. This kernel provides a simple expression for the $c_\omega$-capacity.

	\begin{proposition}\label{Kernel*}
		Let $\omega\in\mathcal W_{\mathrm{Dou}}$ satisfy Condition \eqref{condition1}. Then
		$\mathcal D^h(\omega)^*$ is a  Hilbert space with reproducing kernel given by
		$$
		K_\omega^*(x^*,y^*)
		=
		\int_{x^*}\int_{y^*}
		\left(
		2\operatorname{Re}\bigl(k^\omega(\lambda,\zeta)\bigr)-1
		\right)
		\,d\mu_{x^*}(\lambda)\,
		d\mu_{y^*}(\zeta),
		$$
		for every $x^*,y^*\in\mathcal B^*$.
	\end{proposition}

	\begin{proof}
		Fix $x^*\in\mathcal B^*$. For every
		$u^*\in\mathcal D^h(\omega)^*$ there exists  
		$u\in\mathcal D^h(\omega)$ satisfying
		$$
		u^*(x^*)
		=
		\int_{x^*}u(\lambda)\,
		d\mu_{x^*}(\lambda).
		$$
		Since $\mathcal D^h(\omega)$ is a measurable functional Hilbert space in the sense of Aronszajn and Smith, the map
		$
		u
		\longmapsto
		\displaystyle\int_{x^*}u\,d\mu_{x^*}
		$
		is continuous on $\mathcal D^h(\omega)$. Hence the evaluation functional
		$
		u^*\longmapsto u^*(x^*)
		$
		is continuous on $\mathcal D^h(\omega)^*$.
		By the Riesz representation theorem,
		$\mathcal D^h(\omega)^*$ is therefore a reproducing kernel Hilbert space.
		
		To identify its reproducing kernel, let
		$u^*\in\mathcal D^h(\omega)^*$ correspond to
		$u\in\mathcal D^h(\omega)$. Since
		$\mathcal D^h(\omega)$ is a reproducing kernel Hilbert space,
		$$
		u(\lambda)
		=
		\left\langle
		u,
		\,2\operatorname{Re}\bigl(k^\omega(\cdot,\lambda)\bigr)-1
		\right\rangle_{\mathcal D^h(\omega)}.
		$$
		Consequently
		$$
		\begin{aligned}
			u^*(x^*)
			&=
			\int_{x^*}
			u(\lambda)\,
			d\mu_{x^*}(\lambda)\\
			&=
			\left\langle
			u,
			\int_{x^*}
			\left(
			2\operatorname{Re}\bigl(k^\omega(\cdot,\lambda)\bigr)-1
			\right)
			d\mu_{x^*}(\lambda)
			\right\rangle_{\mathcal D^h(\omega)}.
		\end{aligned}
		$$
		Via the canonical isomorphism between
		$\mathcal D^h(\omega)$ and
		$\mathcal D^h(\omega)^*$, this identity becomes
		$$
		u^*(x^*)
		=
		\langle
		u^*,
		K_\omega^*(\cdot, x^*)
		\rangle_*,
		$$
		where
		$
		K_\omega^*(\cdot, x^*)
		$ is the function in $\mathcal D^h(\omega)^*$ corresponding to $\displaystyle\int_{x^*}
		\left(
		2\operatorname{Re}\bigl(k^\omega(\cdot,\lambda)\bigr)-1
		\right)
		d\mu_{x^*}(\lambda).
		$
		Evaluating this representative at $y^*\in\mathcal B^*$ yields
		$$
		K_\omega^*(x^*,y^*)
		=
		\int_{x^*}\int_{y^*}
		\left(
		2\operatorname{Re}\bigl(k^\omega(\lambda,\zeta)\bigr)-1
		\right)
		d\mu_{x^*}(\lambda)
		d\mu_{y^*}(\zeta),
		$$
		which completes the proof.
	\end{proof}

	\begin{theorem}\label{theorem3}
		Let $\omega\in\mathcal W_{\mathrm{Dou}}$ satisfy Condition \eqref{condition1} and let
		$E$ be a closed subset of $\TT$.
		If $c_\omega(E)>0$, then
		$$
		c_\omega(E)
		=
		\frac{1}{K_\omega^*(E,E)}.
		$$
	\end{theorem}

	\begin{proof}
		Let $\mu_E$ denote the equilibrium measure of $E$ given by Proposition~\ref{proposition41}. Then
		$$
		c_\omega(E)
		=
		\mu_E(E)
		=
		\|p_{\mu_E}\|_{\mathcal D^h(\omega)}^{2}.
		$$
		Define the probability measure
		$$
		\nu_E
		=
		\frac{\mu_E}{\mu_E(E)}.
		$$
		Since
		$$
		p_{\mu_E}
		=
		\mu_E(E)\,p_{\nu_E},
		$$
		it follows that
		$$
		\|p_{\nu_E}\|_{\mathcal D^h(\omega)}^{2}
		=
		\frac{1}{\mu_E(E)}
		=
		\frac{1}{c_\omega(E)}.
		$$
		On the other hand, Proposition~\ref{Kernel*} yields
		$$
		\begin{aligned}
			K_\omega^*(E,E)
			&=
			\int_E\int_E
			\left(
			2\operatorname{Re}\bigl(k^\omega(\lambda,\zeta)\bigr)-1
			\right)
			\,d\nu_E(\lambda)\,d\nu_E(\zeta) \\
			&=
			\|p_{\nu_E}\|_{\mathcal D^h(\omega)}^{2},
		\end{aligned}
		$$
		where the second equality follows from Proposition~\ref{proposition41}. Consequently
		
		$$
		c_\omega(E)
		=
		\frac{1}{K_\omega^*(E,E)}.
		$$
	\end{proof}

\section{Probabilistic aspect}\label{section4}
\subsection{Hunt processes and the spaces $\cD^h(\omega)$}
Let $\omega \in \mathcal W_{\mathrm{Dou}}$. \((\mathcal E_\omega,\mathcal F_\omega)\) is a regular non--local Dirichlet form on \(\IL\) with core \(C^\infty(\TT)\), and \((\mathcal G_\omega,\mathcal D(\mathcal G_\omega))\) denotes its generator, defined by the Friedrichs extension through$$\mathcal{E}_\omega(u,v)=\langle \mathcal{G}_\omega u,v\rangle_{\IL},\qquad u\in\mathcal{D}(\mathcal{G}_\omega),\quad v\in\mathcal{F}_\omega:=\overline{C^\infty(\TT)}^{\|\cdot\|_{\cD^h(\omega)}}.$$

For a more detailed discussion of Dirichlet forms and their connection with stochastic processes, the reader may refer to \cite{fukushima2010dirichlet}.

\begin{definition}
 A Markov process
	$$  \mathbb X =\bigl(\Omega,\mathcal F,(\mathcal F_t)_{t\geq0},\tau,(X_t)_{t\geq0},(\mathbb P_\zeta)_{\zeta\in\TT_\Delta}\bigr)$$
	on $\TT_\Delta:=\TT\cup\{\Delta\}$ is called an $m$-symmetric Hunt process if it is normal, strong Markov, has right-continuous paths with left limits, is quasi-left-continuous, and its transition operators
	$$\mathrm{P}_tf(\zeta):=\mathbb E_\zeta\!\left[f(X_t);\,t<\tau\right]$$
	satisfy
	$$\int_{\TT} \mathrm{P}_tf(\zeta)\,\overline{g(\zeta)}\,dm(\zeta)=\int_{\TT} f(\zeta)\,\overline{\mathrm{P}_tg(\zeta)}\,dm(\zeta)$$
	for every $t>0$ and all bounded Borel functions $f$ and $g$.
\end{definition}

The following result holds.
\begin{theorem} Let $\omega \in \mathcal W_{\mathrm{Dou}}$. Then there exists an $m$-symmetric Hunt process
	$$\mathbb X^\omega=\bigl(\Omega,\mathcal F,(\mathcal F_t)_{t\ge0},\tau^\omega,(X_t^\omega)_{t\ge0},(\mathbb P_\zeta^\omega)_{\zeta\in\TT_\Delta}\bigr)$$ 
	properly associated with $(\mathcal{E}_\omega,\mathcal{F} _\omega)$ on the state space $\TT_\Delta$, 
	where $\Delta$ is the cemetery point. The process satisfies:
	\begin{enumerate}
		\item $\mathbb X^\omega$ is unique up to standard equivalence outside a properly exceptional set.
		\item Lifetime:$$\tau^\omega:=\inf\Big\{t\geq0:X_t^\omega=\Delta\Big\}=\infty\qquad\mathbb P_\zeta^\omega\text{-a.s.}$$for \(c_\omega\)-q.e. \(\zeta\in\TT\).
		\item $\mathbb X^\omega$ is not a diffusion, and its jump behavior is encoded by $\mathbf{J}_\omega$.
	\end{enumerate}
\end{theorem}  
\begin{proof}  
	Existence and proper association follow from \cite[Theorem~7.2.1]{fukushima2010dirichlet}. Uniqueness holds up to standard equivalence. Since the form is conservative, $\tau^\omega=\infty,\, \mathbb P_\zeta^\omega\text{-a.s.},$
	for \(c_\omega\)-q.e. \(\zeta\in\mathbb T\). 
	Finally the Beurling--Deny decomposition contains only the jump part \(\mathbf{J}_\omega\); hence \(\mathbb X^\omega\) is not a diffusion.
\end{proof}  
\begin{definition} A pair \((N_\omega,H_\omega)\) is called a Lévy system of \(\mathbb X^\omega\) if for every $ t>0$  and every nonnegative Borel function \(f\) on \(\TT\times\TT\) vanishing on the diagonal
	$$\mathbb E^\omega_\zeta\left[\sum_{0<s\le t}f\!\left(X^\omega_{s-},X^\omega_s\right)\right]=\mathbb E^\omega_\zeta\left[\int_0^t\int_{\TT}f\!\left(X^\omega_s,\lambda\right)N_\omega\!\left(X^\omega_s,d\lambda\right)\,dH^\omega_s\right]$$   for \(c_\omega\)-q.e. \(\zeta\in\TT\). 
\end{definition}
Let \(\mu_{H^\omega}\) be the Revuz measure of \(H^\omega\), that is, the unique positive smooth measure satisfying
$$\lim_{t\downarrow0}\frac{1}{t}\,\mathbb E^\omega_m\left[\int_0^t f(X^\omega_s)\,dH^\omega_s\right]=\int_{\TT}f(\zeta)\,\mu_{H^\omega}(d\zeta)$$for every nonnegative Borel function \(f\) on \(\TT\).
\begin{proposition}\label{propositionjump}
	Let $\omega \in \mathcal W_{\mathrm{Dou}}$. The  jumping measure is given by
	$$
	\mathbf{J}_\omega(d\zeta,d\lambda)
	=
	\frac12
	N_\omega(\zeta,d\lambda)\,
	\mu_{H^\omega}(d\zeta)
	$$
	on $\TT^2\setminus\operatorname{diag}$.
\end{proposition}
\begin{proof}
	This follows from \cite[Theorem~5.3.1]{fukushima2010dirichlet}.
\end{proof} 
We now characterize the capacity \(c_{\omega}\) in terms of the associated process \(\mathbb X^{\omega}\). For a Borel set \(A\subset\TT\), define the first hitting time by
$$\sigma_A^\omega:=\inf\Big\{t>0:X_t^\omega\in A\Big\}.$$
\begin{theorem}\label{thm5-process}
	Let \(\omega\in\mathcal W_{\mathrm{Dou}}\) and let \(A\subset\TT\) be a Borel set. Then$$\widetilde{p_{\mu_A}}(\zeta)=\mathbb E_\zeta^\omega\left[e^{-\sigma_A^\omega};\,\sigma_A^\omega<\infty\right],\qquad c_\omega\text{-q.e. }\zeta\in\TT.$$
	Moreover$$c_\omega(A)=0\quad\Longleftrightarrow\quad\mathbb P_\zeta^\omega\left(\sigma_A^\omega<\infty\right)=0,\qquad c_\omega\text{-q.e. }\zeta\in\TT.$$
\end{theorem}
\begin{proof}
	The first identity follows from \cite[Theorem~4.2.5]{fukushima2010dirichlet}. Since
	$$c_\omega(A)=0 \iff \widetilde{p_{\mu_A}}=0\quad c_\omega\text{-q.e.},$$
	the second assertion follows from the strict positivity of \(e^{-\sigma_A^\omega}\) on \(\{\sigma_A^\omega<\infty\}\).
\end{proof}
Thus $A$ has zero capacity precisely when it is not hit from quasi-every starting point.

\subsection{Cauchy process and the space $\cD(\TT)$}
In the case \(\omega\equiv1\), for \(f,g\in C^\infty(\TT)\),
$$\mathcal E_1(f,g)=\sum_{n\in\mathbb Z}|n|\,\widehat f(n)\,\overline{\widehat g(n)}=\iint_{\TT^2\setminus\operatorname{diag}}\frac{\bigl(f(\zeta)-f(\lambda)\bigr)\overline{\bigl(g(\zeta)-g(\lambda)\bigr)}}{|\zeta-\lambda|^2}\,dm(\zeta)dm(\lambda).$$
Moreover
$$\mathcal F_1=\left\{f\in \IL:\sum_{n\in\mathbb Z}(1+|n|)\,|\widehat f(n)|^2<\infty\right\}=\cD(\TT).$$
Thus \((\mathcal E_1,\cD(\TT))\) is a regular non--local Dirichlet form on \(\IL\) with core \(C^\infty(\TT)\). Let \(e_n(\zeta)=\zeta^n\). For \(u\in \mathcal{D}(\mathcal{G}_1)\),
$$
\widehat{\mathcal{G}_1u}(n)
=
\langle\mathcal{G}_1u,e_n\rangle_{\IL}
=
\mathcal{E}_1(u,e_n)
=
|n|\,\widehat u(n),$$
and
$$
\mathcal{D}(\mathcal{G}_1)
=
\Big\{
u\in \IL:
\sum_{n\in\mathbb Z}
(1+n^2)|\widehat u(n)|^2<\infty
\Big\}.
$$
The generator is given by
$\mathcal G_1=(-\Delta_{\TT})^{1/2}.$
By analogy with  \cite[Theorem~1.1]{kwasnicki2017ten}, we consider the periodic realizations of \(\mathcal G_1\). For \(\zeta\in\TT\) and \(0<r<\pi\), set
$$\tau_{\zeta,r}:=\inf\left\{t>0:X_t^1\notin \mathrm{b}_{\TT}(\zeta,r)\right\}, \qquad
\mathrm{b}_{\TT}(\zeta,r)
:=
\left\{
\zeta e^{ih}:|h|<r
\right\}.$$
and define$$H_rf(\zeta):=\mathbb E_\zeta^1\left[f\!\left(X_{\tau_{\zeta,r}}^1\right)\right],\qquad a_r:=\mathbb E_\zeta^1[\tau_{\zeta,r}].$$
We have the following theorem.
\begin{theorem}\label{thm6-processus}
	Let \(\mathcal X\) be any of the spaces \(\mathrm{L^p(\TT)}\), \(p\in[1,\infty)\), or \(C(\TT)\) and let \(f\in\mathcal X\). The following definitions of \(\mathcal G_1f\in\mathcal X\) are equivalent:
	\begin{enumerate}
		\item Fourier definition:$$\widehat{\mathcal G_1f}(n)=|n|\,\widehat f(n),\qquad n\in\mathbb Z.$$
		\item Distributional definition:
		$$\int_{\TT}\mathcal G_1f(\zeta)\,\overline{\varphi(\zeta)}\,dm(\zeta)=\int_{\TT}f(\zeta)\,\overline{\mathcal G_1\varphi(\zeta)}\,dm(\zeta)$$for every \(\varphi\in C^\infty(\TT)\).
		\item Bochner definition:
		$$\mathcal G_1f=\frac{1}{2\sqrt{\pi}}\int_0^\infty\bigl(f-e^{t\Delta_{\TT}}f\bigr)\frac{dt}{t^{3/2}}$$with the Bochner's integral of an \(\mathcal X\)-valued function.
		\item Balakrishnan definition:$$\mathcal G_1f=\frac{1}{\pi}\int_0^\infty(-\Delta_{\TT})(sI-\Delta_{\TT})^{-1}f\frac{ds}{\sqrt{s}},$$
		with the Bochner's integral of an \(\mathcal X\)-valued function.
		\item Singular integral definition:
		$$\mathcal G_1f=\lim_{\varepsilon\downarrow0}\frac{1}{4\pi}\int_{\varepsilon<|h|<\pi}\frac{f(\,\cdot\,)-f(\,\cdot\,e^{ih})}{\sin^2(h/2)}\,dh,$$
		with the limit in \(\mathcal X\).
		\item Dynkin definition:
		$$\mathcal G_1f=\lim_{r\downarrow0}\frac{f-H_rf}{a_r}$$with the limit in \(\mathcal X\).
		\item Quadratic form definition:$$\langle \mathcal G_1f,\varphi\rangle_{\IL}=\mathcal E_1(f,\varphi)$$for every \(\varphi\in \cD(\TT)\), when \(\mathcal X=\IL\).
		\item Semigroup definition:$$\mathcal G_1f=\lim_{t\downarrow0}\frac{f-\mathrm{T}_t^1f}{t},$$
		where $\mathrm{T}_t^1f=f*p_t^1$ and $
		\widehat{p_t^1}(n)=e^{-t|n|}.$
		\item Definition as the inverse of the periodic Riesz potential:
		$$\int_{\TT}-2\log |\,\cdot-\lambda\,|\,\mathcal G_1f(\lambda)\,dm(\lambda)=f(\,\cdot\,)-\widehat f(0).$$
		\item Definition through harmonic extensions:$$\mathcal G_1f=\lim_{y\downarrow0}\,\frac{f-u_f(\cdot,y)}{y},$$where $u_f(\cdot,y)=f*\pi_y$ and $\widehat{\pi_y}(n)=e^{-y|n|}$.
	\end{enumerate}
\end{theorem} 
Since Theorem~\ref{thm6-processus} identifies the semigroup \(e^{-t\mathcal G_1}\) with the wrapped Cauchy semigroup, we choose the representative
$$
X_t^{1,\zeta}
:=
\zeta e^{iC_t},
\qquad
t\ge0,\quad \zeta\in\TT,
$$
where \(C:=(C_t)_{t\geq0}\) is a symmetric Cauchy process, equivalently the symmetric strictly \(1\)-stable Lévy process normalized by
$$
\mathbb E\!\left[e^{i\xi C_t}\right]
=e^{-t|\xi|},
\qquad \xi\in\mathbb R.
$$
\begin{theorem}
	The Hunt process associated with
	\(\bigl(\mathcal E_1,\cD(\TT)\bigr)\) is the wrapped Cauchy process
	\((X_t^{1,\zeta})_{t\ge0}\), which satisfies:
	
	\begin{enumerate}
		\item Transition law:
		$$
		\mathbb P_\zeta^1
		\bigl(X_t^{1,\zeta}\in d\lambda\bigr)
		=
		\frac{1-e^{-4t}}
		{|\lambda-\zeta e^{-2t}|^2}\,dm(\lambda),\quad t>0.
		$$
		
		\item Lévy measure:
		$$
		\nu_1(d\lambda)
		=
		\frac{4}{|1-\lambda|^2}\,dm(\lambda),
		\qquad
		\lambda\in\TT\setminus\{1\}.
		$$
		
		\item Lévy system:
		$$
		\left(\nu_1\!\left(d(\overline{\zeta}\lambda)\right),t\right),
		\quad
		\nu_1\!\left(d(\overline{\zeta}\lambda)\right)
		=
		\frac{4}{|\zeta-\lambda|^2}\,dm(\lambda), \qquad \lambda\not=\zeta.
		$$
		
		\item For every Borel set \(A\subset\TT\),
		$$
		c_1(A)=0
		\quad\Longleftrightarrow\quad
		\mathbb P_\zeta^1\!\left(\sigma_A^1<\infty\right)=0,
		\qquad
		c_1\text{-q.e. }\zeta\in\TT,
		$$
		where
		$$
		\sigma_A^1
		:=
		\inf\left\{t>0:X_t^{1,\zeta}\in A\right\}.
		$$
	\end{enumerate}
\end{theorem}
\begin{proof}
	The semigroup identity follows from$$\mathbb E_{\zeta}^{1}\!\left[\left(X_t^{1,\zeta}\right)^n\right]=e^{-t|n|}\zeta^n.$$Hence the semigroups coincide on trigonometric polynomials. Density and contractivity extend the identity to \(\IL\). The transition law and the Lévy measure follow by periodization. Since \(H_t=t\) has Revuz measure \(m\), Proposition~\ref{propositionjump} yields the Lévy system. The last assertion follows from Theorem~\ref{thm5-process}.
\end{proof}
	\section{Open problems}\label{section5}
	
	The results obtained in this paper suggest several natural directions for future investigation, lying at the intersection of function theory, Dirichlet forms, potential theory, and stochastic processes.

	\begin{problem}
		Characterize the class of weights for which the weighted Dirichlet integral admits a Douglas--type boundary representation.
	\end{problem}
	
	Another fundamental question concerns the relationship between reproducing kernels and capacity. The present work establishes this correspondence for Douglas weights, and it is natural to ask whether it persists in a more general setting.
	
	\begin{problem}
		Does the equivalence between the diagonal growth of the reproducing kernel and the associated capacity extend to arbitrary positive weights on the unit disk?
	\end{problem}
	
	A further direction is to investigate the potential--theoretic structure of exceptional sets associated with weighted Dirichlet spaces. In particular, one may ask whether the geometry of $c_\omega$-polar sets is encoded in the asymptotic behavior of the reproducing kernel.
	
	\begin{problem}
		Can $c_\omega$-polar sets be characterized in terms of the asymptotic behavior of the diagonal reproducing kernel
		$$
		k^\omega(r\zeta,r\zeta),
		\qquad r\to1^-,
		$$
		together with suitable geometric or metric properties of subsets of $\mathbb T$?
	\end{problem}
	
	From the probabilistic point of view, it would be interesting to obtain a more explicit description of the Hunt processes associated with Douglas--type weights. In particular, one may ask how the geometry of the weight influences the Lévy system, the jump kernel, and the sample--path behavior of the associated process.
	
	\begin{problem}
		Describe the infinitesimal generator, the Lévy measure, and the jump kernel of the Hunt process associated with a general Douglas--type weight.
	\end{problem}
	
	Another natural direction is to investigate finer boundary regularity and sharp estimates for reproducing kernels and capacities. Such estimates could provide a better understanding of cyclicity, uniqueness sets, and boundary zero sets in weighted Dirichlet spaces.
	
	\begin{problem}
		Establish sharp asymptotic estimates for reproducing kernels and capacities associated with Douglas--type weights and determine their consequences for cyclicity and uniqueness phenomena.
	\end{problem}
	
	It is also natural to ask whether the present framework extends beyond Douglas--type weights.
	
	\begin{problem}
		Develop a Dirichlet form and Hunt process theory for more general weighted Dirichlet spaces, including weights that do not admit a Douglas--type boundary representation.
	\end{problem}
	
	Finally the probabilistic realization developed in this paper raises the question of whether classical problems in weighted Dirichlet spaces admit natural stochastic interpretations. A particularly intriguing example is the Brown--Shields conjecture.
	
	\begin{problem}
		Does the Brown--Shields conjecture admit a stochastic formulation or an equivalent probabilistic characterization in terms of the Hunt process associated with the underlying weighted Dirichlet space?
	\end{problem}
	\appendix    
	\section{}\label{section6}
	\subsection{The Brownian Dirichlet form}
	
	We consider the Hardy--Sobolev space
	
	$$
	H^2_1
	=
	\left\{
	f(z)=\sum_{n=0}^{\infty}a_nz^n:
	\sum_{n=0}^{\infty}(1+n^2)|a_n|^2<\infty
	\right\},
	$$
	or equivalently
	
	$$
	H^2_1
	=
	\{f\in\Hol(\DD):f'\in \hH\}.
	$$
	Its harmonic version is the Sobolev space
	
	$$
	H^1(\TT)
	=
	\left\{
	u\in \IL:
	\sum_{n\in\mathbb Z}(1+n^2)|\widehat u(n)|^2<\infty
	\right\}.
	$$
	The associated regular Dirichlet form is
	
	$$
	\E_B(u,v)
	=
	\int_{0}^{2\pi}
	u'(\theta)\overline{v'(\theta)}
	\,d\theta,
	$$
	with domain
	
	$$
	\cF_B=H^1(\TT).
	$$
	Equivalently
	
	$$
	\E_B(u,u)
	=
	\sum_{n\in\mathbb Z}
	n^2|\widehat u(n)|^2.
	$$
	The associated Hunt process is the one--dimensional Brownian motion on
	$\TT$. The associated generator is given by
	
	$$
	\mathcal G_B=-\Delta_\TT.
	$$
	Since Brownian trajectories are continuous, every point of $\TT$ is hit almost surely. Consequently if $\mathrm{Cap}_B$ denotes the capacity associated with the Brownian Dirichlet form, then
	$$
	\mathrm{Cap}_B(\{\zeta\})>0,
	\qquad
	\zeta\in\TT.
	$$
	Equivalently there are no non--empty polar sets.


	\subsubsection{Comparison of the Brownian and Cauchy frameworks}
	
	\begin{center}
		\renewcommand{\arraystretch}{1.5}
		\begin{tabular}{p{5cm}cc}
			\toprule
			& Brownian motion & Cauchy process\\
			\midrule
			
			Analytic space
			&
			$H^2_1=\{f:f'\in \hH\}$
			&
			$\cD=\{f:f'\in A^2(\DD)\}$
			\\
			
			Harmonic space
			&
			$H^1(\TT)$
			&
			$\cD(\TT)$
			\\
			
			Dirichlet form
			&
			$\displaystyle
			\int_{\TT}|u'(\zeta)|^2dm(\zeta)
			$
			&
			$\displaystyle
			\iint_{\TT^2}
			\frac{|u(\zeta)-u(\lambda)|^2}
			{|\zeta-\lambda|^2} dm(\zeta)dm(\lambda)
			$
			\\
			
			Fourier symbol
			&
			$n^2$
			&
			$|n|$
			\\
			
			Generator
			&
			$-\Delta_\TT$
			&
			$\big(-\Delta_\TT\big)^{1/2}$
			\\
			
			Nature
			&
			Local
			&
			Non--local
			\\
			
			Sample--paths
			&
			Continuous
			&
			C\`adl\`ag with jumps
			\\
			
			Capacity of a singleton
			&
			Positive
			&
			Zero
			\\
			
			Polar sets
			&
			Only the empty set
			&
			Sets of logarithmic capacity zero
			\\
			
			Probability of hitting one point
			&
			$1$
			&
			$0$
			\\
			
			\bottomrule
		\end{tabular}
	\end{center}
	\subsection{Fractional Dirichlet forms}
	
	Let $0<\alpha\le2$. We consider the quadratic form
	
	$$
	\E_\alpha(f,g)
	=
	\sum_{n\in\mathbb Z}
	|n|^\alpha
	\widehat f(n)
	\overline{\widehat g(n)}, \quad f,g\in\IL,
	$$
	with domain
	
	$$
	\cF_\alpha
	=
	\left\{
	f\in \IL:
	\sum_{n\in\mathbb Z}
	(1+|n|^\alpha)
	|\widehat f(n)|^2<\infty
	\right\},
	$$
	which is the fractional Sobolev space of order $\alpha/2$.
	
	The form $(\mathcal{E}_\alpha,\mathcal{F}_\alpha)$ is a regular symmetric Dirichlet form on $\IL$.

	
	\subsubsection{Generator}
	The infinitesimal generator associated with 
	$(\mathcal{E}_\alpha,\mathcal{F}_\alpha)$ is denoted by
	$\mathcal{G}_\alpha$. It is given by
	$$
	\mathcal{G}_\alpha=(-\Delta_\TT)^{\alpha/2},
	$$
	or equivalently
	$$
	\widehat{\mathcal{G}_\alpha f}(n)
	=|n|^\alpha \widehat{f}(n),
	\qquad n\in\mathbb{Z}.
	$$
	Hence
	$$
	\mathcal{E}_\alpha(f,g)
	=
	\langle\mathcal{G}_\alpha f,g\bigr\rangle_{\IL},
	\qquad
	f\in \mathcal D(\mathcal{G}_\alpha),\;
	g\in\mathcal{F}_\alpha.
	$$
	\subsubsection{Resolvent}
	
	The resolvent operator associated with $\mathcal{G}_\alpha$ is
	
	$$
	\mathcal R_\alpha
	=
	\big(1+(-\Delta_\TT)^{\alpha/2}\big)^{-1},
	$$
	whose Fourier multiplier is
	
	$$
	\widehat{\mathcal R_\alpha f}(n)
	=
	\frac{\widehat f(n)}
	{1+|n|^\alpha}, \quad n\in\ZZ,
	$$
	and $$\mathcal R_\alpha\big(e^{in\theta}\big)=\frac{e^{in\theta}}
	{1+|n|^\alpha}.$$
	\subsubsection{Integral Representation}
	
	For $0<\alpha<2$, the Dirichlet form admits the non--local representation
	
	$$
	\E_\alpha(f,f)
	=
	c_\alpha
	\iint_{\TT\times\TT}
	\frac{|f(\zeta)-f(\eta)|^2}
	{|\zeta-\eta|^{1+\alpha}}
	\,dm(\zeta)\,dm(\eta),
	$$ 
    where $c_\alpha>0$ is a normalization constant.
    \begin{itemize}
        \item For $\alpha=1$, this becomes the classical Douglas form
	
	$$
	\E_1(f,f)
	=
	\iint_{\TT\times\TT}
	\frac{|f(\zeta)-f(\eta)|^2}
	{|\zeta-\eta|^2}
	\,dm(\zeta)\,dm(\eta).
	$$
    \item For $\alpha=2$, the form becomes local,
	
	$$
	\E_2(f,f)
	=
	\int_{0}^{2\pi}
	|f'(\theta)|^2\,d\theta.
	$$
    \end{itemize}
	\subsubsection{Holomorphic counterpart}
	
	The corresponding analytic function space of $(\mathcal{E}_\alpha, \mathcal{F}_\alpha)$ is
	
	$$
	\cD_\alpha
	=
	\left\{
	f(z)
	=
	\sum_{n=0}^\infty
	a_nz^n:
	\sum_{n=0}^\infty
	(1+n^\alpha)|a_n|^2
	<
	\infty
	\right\}.
	$$
	Special cases include
	
	$$
	\begin{aligned}
		\alpha=0
		&\quad\Longrightarrow\quad
		\hH,
		\\
		\alpha=1
		&\quad\Longrightarrow\quad
		\cD,
		\\
		\alpha=2
		&\quad\Longrightarrow\quad
		H^2_1.
	\end{aligned}
	$$
	\subsubsection{The reproducing kernel of $\cD_\alpha$ and the resolvent}
	The reproducing kernel of the space $\cD_\alpha$ is obtained from the
	resolvent representation. More precisely, for $w=r\zeta\in\DD$ and
	$\lambda\in\TT$, we have
	$$
	\begin{aligned}
		\mathcal R_\alpha k^0_w(\lambda)
		&=
		\sum_{n=0}^{\infty}
		r^n\,\mathcal{R}_\alpha\!\left((\overline{\lambda}\zeta)^n\right) \\
		&=
		\sum_{n=0}^{\infty}
		\frac{r^n(\overline{\lambda}\zeta)^n}{1+n^\alpha}\\
        &=k_w^\alpha(\lambda).
	\end{aligned}
	$$
	This is precisely the reproducing kernel of the fractional Dirichlet space
	$\cD_\alpha$.
	\subsubsection{Associated Hunt process}
	
	By the Fukushima correspondence, there exists a unique symmetric Hunt
	process associated with $(\E_\alpha,\cF_\alpha)$. This process is precisely the symmetric $\alpha$-stable process on $\TT$.
	\begin{itemize}
	    \item For $\alpha=2$, we recover the one--dimensional Brownian motion.
	
	\item For $\alpha=1$, we recover the wrapped Cauchy process.
	\end{itemize}

	\subsubsection{Sample--paths}
	
	The nature of the sample--paths also depends on $\alpha$.
	
	\begin{center}
		\begin{tabular}{cc}
			\toprule
			$\alpha$ & Sample--paths\\
			\midrule
			$2$ & Continuous\\
			$0<\alpha<2$ & C\`adl\`ag with jumps\\
			\bottomrule
		\end{tabular}
	\end{center}
	
	Thus the Brownian motion is the unique diffusion in this family,
	whereas every stable process with $0<\alpha<2$ is purely discontinuous.
	
	
	\subsubsection{Local versus non--local forms}
	
	The fractional Dirichlet forms satisfy
	
	$$
	\begin{cases}
		\alpha=2,
		&
		\text{local Dirichlet form},
		\\
		0<\alpha<2,
		&
		\text{non--local Dirichlet form}.
	\end{cases}
	$$
	
	The transition from $\alpha=2$ to $0<\alpha<2$ corresponds to replacing
	continuous trajectories by jump processes.


	
	\subsubsection{Capacities and polar sets}
	
	The Dirichlet form defines a capacity denoted by
	\(
	c_\alpha.
	\)
	For every Borel set $A\subset\TT$,
	
	$$
	c_\alpha(A)=0
	\Longleftrightarrow
	\mathbb P_\zeta
	(\sigma_A<\infty)=0,
	\qquad
	c_\alpha
	\text{-q.e. }\zeta.
	$$
	
	Thus the sets of zero capacity coincide with the polar sets of the
	associated symmetric stable process. In dimension one, the behavior of singletons changes at the critical
	value $\alpha=1$.
	
	$$
	c_\alpha(\{\zeta\})
	=
	0
	\iff
	0<\alpha\le1.
	$$

	\subsubsection{Summary tables for fractional Dirichlet forms}
	\begin{table}[ht]
		\centering
		\renewcommand{\arraystretch}{1.4}
		\begin{tabular}{p{4.2cm}p{10.5cm}}
			\toprule
			\textbf{Property} & \textbf{Description}\\
			\midrule
			
			Dirichlet form &
			$\displaystyle
			\mathcal{E}_\alpha(f,g)
			=
			\sum_{n\in\mathbb Z}
			|n|^\alpha
			\widehat f(n)\overline{\widehat g(n)}.
			$
			\\
			
			Domain &
			$\displaystyle
			\mathcal{F}_\alpha
			=
			\Bigl\{
			f\in \IL:
			\sum_{n\in\mathbb Z}
			(1+|n|^\alpha)
			|\widehat f(n)|^2<\infty
			\Bigr\}.
			$
			\\
			
			Type &
			Regular symmetric Dirichlet form on $\IL$.
			\\
			
			Generator &
			$\displaystyle
			\mathcal G_\alpha=(-\Delta_{\mathbb T})^{\alpha/2},
			\qquad
			\widehat{\mathcal G_\alpha f}(n)
			=
			|n|^\alpha\widehat f(n).
			$
			\\
			
			Energy identity &
			$\displaystyle
			\mathcal E_\alpha(f,g)
			=
			\langle\mathcal G_\alpha f,g\rangle_{\IL}.
			$
			\\
			
			Resolvent &
			$\displaystyle
			\mathcal R_\alpha
			=(I+\mathcal G_\alpha)^{-1},
			\qquad
			\widehat{\mathcal R_\alpha f}(n)
			=
			\frac{\widehat f(n)}
			{1+|n|^\alpha}.
			$
			\\
			
			Integral representation &
			For $0<\alpha<2$,
			$$
			\mathcal E_\alpha(f,f)
			=
			c_\alpha
			\iint_{\mathbb T^2}
			\frac{|f(\zeta)-f(\lambda)|^2}
			{|\zeta-\lambda|^{1+\alpha}}
			\,dm(\zeta)\,dm(\lambda).
			$$
			For $\alpha=2$,
			$\displaystyle
			\mathcal E_2(f,f)
			=
			\int_{0}^{2\pi}|f'(\theta)|^2d\theta.
			$
			\\
			
			Analytic space &
			$\displaystyle
			\mathcal D_\alpha
			=
			\Bigl\{
			f(z)=\sum_{n\ge0}a_nz^n:
			\sum_{n\ge0}(1+n^\alpha)|a_n|^2<\infty
			\Bigr\}.
			$
			\\
			
			Special cases &
			$\alpha=0:\ \hH,$
			\qquad
			$\alpha=1:\ \mathcal D,$
			\qquad
			$\alpha=2:\ H^2_1.$
			\\
			
			Reproducing kernel &
			$\displaystyle
			k_w^\alpha(\lambda)
			=\mathcal R_\alpha k^0_w(\lambda)=
			\sum_{n=0}^{\infty}
			\frac{(\lambda\overline w)^n}
			{1+n^\alpha},
			\qquad
			w=r\zeta.
			$
			\\
			
			Associated Hunt process &
			Symmetric $\alpha$-stable process on $\mathbb T$.
			\\
			
			Special processes &
			$\alpha=2$: Brownian motion.
			
			$\alpha=1$: Cauchy process.
			\\
			
			Sample--paths &
			Continuous if $\alpha=2$.
			
			Càdlàg with jumps if $0<\alpha<2$.
			\\
			
			Locality &
			Local for $\alpha=2$.
			
			Non--local for $0<\alpha<2$.
			\\
			
			Capacity &
			$c_\alpha(A)=0$
			iff
			$A$ is polar for the associated Hunt process.
			\\
			
			Singletons &
			$\displaystyle
			c_\alpha(\{\zeta\})=0
			\iff
			0<\alpha\le1.
			$
			\\
			
			Critical exponent &
			$\alpha=1$ separates polar and non--polar singletons.
			\\
			
			\bottomrule
		\end{tabular}
		\caption{Main properties of the fractional Dirichlet forms
			$(\mathcal E_\alpha,\mathcal F_\alpha)$, $0<\alpha\le2$.}
	\end{table}
	\FloatBarrier
	\subsection*{Acknowledgements}The first and fourth authors gratefully acknowledge the support of the National Centre for Scientific and Technical Research (CNRST), Morocco, through the PhD Associate Scholarship Program (PASS).
	\bibliographystyle{abbrv}
	\bibliography{bibliography}

@inproceedings{abakumov2015cyclicity,
  title={Cyclicity in the harmonic {D}irichlet space},
  author={Abakumov, Evgueni and El-Fallah, Omar and Kellay, Karim and Ransford, Thomas},
  booktitle={Conference on Harmonic and Functional, Analysis, Operator Theory and Applications. 1--10, Theta Ser. Adv. Math., 19, Theta, Bucharest, 2017.},
  year={2015}
}

@phdthesis{aleman1993multiplication,
  author = {Aleman, Alexandru},
  title  = {The Multiplication operator on Hilbert spaces of analytic functions},
  school = {FernUniversität in Hagen},
  year   = {1993},
}

@book {ARSW,
    AUTHOR = {Arcozzi, Nicola and Rochberg, Richard and Sawyer, Eric T. and
              Wick, Brett D.},
     title = {The {D}irichlet space and related function spaces},
    SERIES = {Mathematical Surveys and Monographs},
    VOLUME = {239},
 PUBLISHER = {American Mathematical Society, Providence, RI},
      YEAR = {2019},
     PAGES = {xix+536},
      ISBN = {978-1-4704-5082-3},
   MRCLASS = {31-02 (30-02 30Hxx 32-02 39-02 47-02)},
  MRNUMBER = {3969961},
MRREVIEWER = {H. Turgay Kaptano\u{g}lu},
       DOI = {10.1090/surv/239},
       URL = {https://doi.org/10.1090/surv/239},
}

@article{aronszajn1957characterization,
  title={Characterization of positive reproducing kernels. Applications to Green's functions},
  author={Aronszajn, Nachman and Smith, Kennan Taylor},
  journal={American Journal of Mathematics},
  volume={79},
  number={3},
  pages={611--622},
  year={1957},
  publisher={JSTOR}
}

@book{axler2001harmonic,
  author    = {Axler, Sheldon and Bourdon, Paul and Ramey, Wade},
  title     = {Harmonic function theory},
  edition   = {2},
  series    = {Graduate Texts in Mathematics},
  volume    = {137},
  publisher = {Springer Science \& Business Media},
  address   = {New York},
  year      = {2001},
  isbn      = {978-0-387-95218-5},
  doi       = {10.1007/978-1-4757-8137-3},
  mrnumber  = {1805196}
}

@article{bahajji2020approximation,
  title={Approximation in spaces of analytic functions},
  author={Bahajji-El Idrissi, Hafid and El-Fallah, Omar},
  journal={Studia Mathematica},
  volume={255},
  pages={209--217},
  year={2020},
  publisher={Instytut Matematyczny Polskiej Akademii Nauk}
}

@article{idrissi2020blaschke,
  title={Blaschke products and zero sets in weighted {D}irichlet spaces},
  author={Bahajji-El Idrissi, Hafid and El-Fallah, Omar},
  journal={Potential Analysis},
  volume={53},
  number={4},
  pages={1299--1316},
  year={2020},
  publisher={Springer}
}

@article{idrissi2025douglas,
  title={Douglas-type formula for weighted {B}esov spaces},
  author={Bahajji-El Idrissi, Hafid and El-Fallah, Omar},
  journal={Collectanea Mathematica},
  pages={1--17},
  year={2025},
  publisher={Springer}
}

@article{bahajji2026superharmonically,
  title={Superharmonically weighted {D}irichlet spaces},
  author={Bahajji-El Idrissi, H and El-Fallah, O and Elmadani, Y and Hanine, A},
  journal={arXiv preprint arXiv:2605.13787},
  year={2026}
}

@article{bao2018dirichlet,
  title={On {D}irichlet spaces with a class of superharmonic weights},
  author={Bao, Guanlong and G{\"o}{\u{g}}{\"u}{\c{s}}, Nihat Gokhan and Pouliasis, Stamatis},
  journal={Canadian Journal of Mathematics},
  volume={70},
  number={4},
  pages={721--741},
  year={2018},
  publisher={Cambridge University Press}
}

@book{bertoin1996levy,
  author    = {Bertoin, Jean},
  title     = {L{\'e}vy processes},
  series    = {Cambridge Tracts in Mathematics},
  volume    = {121},
  publisher = {Cambridge University Press},
  address   = {Cambridge},
  year      = {1996},
  pages     = {x + 265},
  isbn      = {978-0-521-56243-0},
  isbn13    = {978-0521562430},
  mrnumber  = {1406564},
  doi       = {10.1017/CBO9780511809781}
}

@article{beurling1958espaces,
  title={Espaces de {D}irichlet: I. Le cas {\'e}l{\'e}mentaire},
  author={Beurling, A and Deny, J},
  journal={Acta Mathematica},
  volume={99},
  number={1},
  pages={203--224},
  year={1958},
  publisher={Springer}
}

@book{brelot2011potential,
  editor    = {Brelot, Marcel},
  title     = {Potential theory},
  subtitle  = {Lectures Given at a Summer School of the Centro Internazionale Matematico Estivo (C.I.M.E.) Held in Stresa (Varese), Italy, July 2--10, 1969},
  series    = {C.I.M.E. Summer Schools},
  volume    = {49},
  publisher = {Springer},
  address   = {Berlin, Heidelberg},
  year      = {2011},
  isbn      = {978-3-642-11083-2},
  doi       = {10.1007/978-3-642-11084-9}
}

@article{ChevrotGuillotRansford2010,
  author  = {Nicolas Chevrot and Dominique Guillot and Thomas Ransford},
  title   = {De {B}ranges--{R}ovnyak spaces and {D}irichlet spaces},
  journal = {Journal of Functional Analysis},
  volume  = {259},
  number  = {9},
  pages   = {2366--2383},
  year    = {2010},
  publisher={Elsevier},
  doi     = {10.1016/j.jfa.2010.06.018}
}

@article{diaconis2002gh,
  title={GH {H}ardy and probability???},
  author={Diaconis, Persi},
  journal={Bulletin of the London Mathematical Society},
  volume={34},
  number={4},
  pages={385--402},
  year={2002},
  publisher={Cambridge University Press}
}

@article{el2016cyclicity,
  title={Cyclicity and invariant subspaces in {D}irichlet spaces},
  author={El-Fallah, Omar and Elmadani, Youssef and Kellay, Karim},
  journal={Journal of Functional Analysis},
  volume={270},
  number={9},
  pages={3262--3279},
  year={2016},
  publisher={Elsevier}
}

@article{el2019kernel,
  title={Kernel and capacity estimates in {D}irichlet spaces},
  author={El-Fallah, O and Elmadani, Y and Kellay, K},
  journal={Journal of Functional Analysis},
  volume={276},
  number={3},
  pages={867--895},
  year={2019},
  publisher={Elsevier}
}

@article{el2022extremal,
  title={Extremal functions and invariant subspaces in {D}irichlet spaces},
  author={El-Fallah, O and Elmadani, Y and Labghail, I},
  journal={Advances in Mathematics},
  volume={408},
  pages={108604},
  year={2022},
  publisher={Elsevier}
}

@article{el2024ahlfors,
  title={Ahlfors--David Regular Sets, Point Spectrum and {D}irichlet Spaces},
  author={El-Fallah, O and Elmadani, Y and Labghail, I},
  journal={Results in Mathematics},
  volume={79},
  number={2},
  pages={74},
  year={2024},
  publisher={Springer}
}

@article{el2016dirichlet,
  title={{D}irichlet spaces with superharmonic weights and {D}e {B}ranges--{R}ovnyak spaces},
  author={El-Fallah, Omar and Kellay, Karim and Klaja, Hubert and Mashreghi, Javad and Ransford, Thomas},
  journal={Complex Analysis and Operator Theory},
  volume={10},
  number={1},
  pages={97--107},
  year={2016},
  publisher={Springer}
}

@book{el2014primer,
  author    = {El-Fallah, Omar and Kellay, Karim and Mashreghi, Javad and Ransford, Thomas},
  title     = {A primer on the {D}irichlet space},
  series    = {Cambridge Tracts in Mathematics},
  volume    = {203},
  publisher = {Cambridge University Press},
  address   = {Cambridge},
  year      = {2014},
  isbn      = {978-1-107-04752-5},
  doi       = {10.1017/CBO9781107239425},
  lccn      = {2013372391},
  pages      = {xiii + 211},
  mrnumber   = {3184178}
}

@article{el2009brown,
  title={On the {B}rown--{S}hields conjecture for cyclicity in the {D}irichlet space},
  author={El-Fallah, Omar and Kellay, Karim and Ransford, Thomas},
  journal={Advances in Mathematics},
  volume={222},
  number={6},
  pages={2196--2214},
  year={2009},
  publisher={Elsevier}
}

@article{elmadani2019cyclicity,
  title={Cyclicity in {D}irichlet spaces},
  author={Elmadani, Youssef and Labghail, Imane},
  journal={Canadian Mathematical Bulletin},
  volume={62},
  number={2},
  pages={247--257},
  year={2019},
  publisher={Canadian Mathematical Society}
}

@book{fukushima2010dirichlet,
  author    = {Fukushima, Masatoshi and Oshima, Yoichi and Takeda, Masayoshi},
  title     = {{D}irichlet forms and symmetric {M}arkov processes},
  series    = {De Gruyter Studies in Mathematics},
  volume    = {19},
  edition   = {2},
  publisher = {Walter de Gruyter},
  address   = {Berlin/New York},
  year      = {2011},
  pages     = {viii + 489},
  isbn      = {978-3-11-021808-4},
  eisbn     = {978-3-11-021809-1},
  doi       = {10.1515/9783110218091},
  lccn      = {2010041939},
  mrnumber  = {2778606},
  zbl       = {1225.31001}
}

@article{gallardo2025zero,
  title={Zero-free regions of the {R}iemann zeta function and approximation in weighted {D}irichlet spaces},
  author={Gallardo-Guti{\'e}rrez, Eva A and Seco, Daniel},
  journal={Complex Analysis and Operator Theory},
  volume={19},
  number={2},
  pages={38},
  year={2025},
  publisher={Springer}
}

@phdthesis{guillot2010comportement,
  title={Comportement au bord dans les espaces de {D}irichlet avec poids harmoniques et espaces de {D}e {B}ranges-{R}ovnyak},
  author={Guillot, Dominique},
  year={2010},
  school={Universit{\'e} Laval}
}

@article{guillot2012fine,
  title={Fine boundary behavior and invariant subspaces of harmonically weighted {D}irichlet spaces},
  author={Guillot, Dominique},
  journal={Complex Analysis and Operator Theory},
  volume={6},
  number={6},
  pages={1211--1230},
  year={2012},
  publisher={Springer}
}

@book{koosis1998introduction,
  author    = {Koosis, Paul},
  title     = {Introduction to {$H^p$} spaces},
  series    = {Cambridge Tracts in Mathematics},
  volume     = {115},
  edition    = {2},
  publisher  = {Cambridge University Press},
  address    = {Cambridge},
  year       = {1998},
  pages      = {xiv + 287},
  isbn       = {978-0-521-45521-3},
  doi        = {10.1017/CBO9780511470950},
  mrnumber   = {1669574}
}

@article{kwasnicki2017ten,
  title={Ten equivalent definitions of the fractional {L}aplace operator},
  author={Kwa{\'s}nicki, Mateusz},
  journal={Fractional Calculus and Applied Analysis},
  volume={20},
  number={1},
  pages={7--51},
  year={2017},
  publisher={De Gruyter}
}

@article{richter1991representation,
  title={A representation theorem for cyclic analytic two-isometries},
  author={Richter, Stefan},
  journal={Transactions of the American Mathematical Society},
  volume={328},
  number={1},
  pages={325--349},
  year={1991}
}

@article{ross1994hyperinvariant,
  title={Hyperinvariant subspaces of the harmonic {D}irichlet space},
  author={Ross, William T and Richter, Stefan and Sundberg, Carl},
  journal={J. reine angew. Math},
  volume={448},
  pages={1--26},
  year={1994},
  publisher={Citeseer}
}

@article{Sarason1997,
  author  = {Donald Sarason},
  title   = {Local {D}irichlet spaces as {D}e {B}ranges--{R}ovnyak spaces},
  journal = {Proceedings of the American Mathematical Society},
  volume  = {125},
  number  = {7},
  pages   = {2133--2139},
  year    = {1997}
}

@article{shimorin2001reproducing,
  title={Reproducing kernels and extremal functions in {D}irichlet--type spaces},
  author={Shimorin, Serguei},
  journal={Journal of Mathematical Sciences},
  volume={107},
  number={4},
  pages={4108--4124},
  year={2001},
  publisher={Springer}
}

@article{shimorin2002complete,
  title={Complete {N}evanlinna--{P}ick property of {D}irichlet--type spaces},
  author={Shimorin, Serguei},
  journal={Journal of Functional Analysis},
  volume={191},
  number={2},
  pages={276--296},
  year={2002},
  publisher={Elsevier}
}

@article{stegenga1980multipliers,
  title={Multipliers of the {D}irichlet space},
  author={Stegenga, David A},
  journal={Illinois Journal of Mathematics},
  volume={24},
  number={1},
  pages={113--139},
  year={1980},
  publisher={Duke University Press}
}

@article{wu2021comprehensive,
  author  = {Zonghan Wu and Shirui Pan and Fengwen Chen and Guodong Long and Chengqi Zhang and Philip S. Yu},
  title   = {A Comprehensive Survey on Graph Neural Networks},
  journal = {IEEE Transactions on Neural Networks and Learning Systems},
  volume  = {32},
  number  = {1},
  pages   = {4--24},
  year    = {2021}
}
\end{document}